\documentclass[12pt]{amsart}
\usepackage{amsmath,amssymb,amsbsy,amsfonts,amsthm,latexsym,amsopn,amstext,amsxtra,euscript,amscd,mathrsfs,color,bm}
\usepackage{float} 
\usepackage[english]{babel}
\usepackage{mathtools}
\usepackage[textsize=tiny]{todonotes}
\usepackage{url}
\usepackage[colorlinks,linkcolor=blue,anchorcolor=blue,citecolor=blue,backref=page]{hyperref}
\usepackage{nicefrac}
\usepackage{enumitem}

\usepackage[a4paper]{geometry}

\usepackage[norefs,nocites]{refcheck}
\usepackage[norefs,nocites]{refcheck}
     
  \usepackage{soul}
\begin{document}
%
    
\newtheorem{theorem}{Theorem}
\newtheorem{lemma}[theorem]{Lemma}
\newtheorem{example}[theorem]{Example}
\newtheorem{algol}{Algorithm}
\newtheorem{corollary}[theorem]{Corollary}
\newtheorem{prop}[theorem]{Proposition}
\newtheorem{proposition}[theorem]{Proposition}
\newtheorem{problem}[theorem]{Problem}
\newtheorem{conj}[theorem]{Conjecture}

\theoremstyle{remark}
\newtheorem{definition}[theorem]{Definition}
\newtheorem{question}[theorem]{Question}
\newtheorem*{acknowledgement}{Acknowledgements}

\newtheorem*{Thm*}{Theorem}
\newtheorem{Thm}{Theorem}[section]
\renewcommand*{\theThm}{\Alph{Thm}}

\numberwithin{equation}{section}
\numberwithin{theorem}{section}
\numberwithin{table}{section}
\numberwithin{figure}{section}

\allowdisplaybreaks

\definecolor{olive}{rgb}{0.3, 0.4, .1}
\definecolor{dgreen}{rgb}{0.,0.5,0.}

\def\cc#1{\textcolor{red}{#1}} 

\definecolor{dgreen}{rgb}{0.,0.6,0.}
\def\tgreen#1{\begin{color}{dgreen}{\it{#1}}\end{color}}
\def\tblue#1{\begin{color}{blue}{\it{#1}}\end{color}}
\def\tred#1{\begin{color}{red}#1\end{color}}
\def\tmagenta#1{\begin{color}{magenta}{\it{#1}}\end{color}}
\def\tNavyBlue#1{\begin{color}{NavyBlue}{\it{#1}}\end{color}}
\def\tMaroon#1{\begin{color}{Maroon}{\it{#1}}\end{color}}

\def\ccr#1{\textcolor{red}{#1}}
\def\ccm#1{\textcolor{magenta}{#1}}
\def\cco#1{\textcolor{orange}{#1}}

%


 \def\mand{\qquad\mbox{and}\qquad}

\def\cA{{\mathcal A}}
\def\cB{{\mathcal B}}
\def\cC{{\mathcal C}}
\def\cS{{\mathcal D}}
\def\cH{{\mathcal H}}
\def\cI{{\mathcal I}}
\def\cJ{{\mathcal J}}
\def\cK{{\mathcal K}}
\def\cL{{\mathcal L}}
\def\cM{{\mathcal M}}
\def\cN{{\mathcal N}}
\def\cO{{\mathcal O}}
\def\cP{{\mathcal P}}
\def\cQ{{\mathcal Q}}
\def\cR{{\mathcal R}}
\def\cS{{\mathcal S}}
\def\cT{{\mathcal T}}
\def\cU{{\mathcal U}}
\def\cV{{\mathcal V}}
\def\cW{{\mathcal W}}
\def\cX{{\mathcal X}}
\def\cY{{\mathcal Y}}
\def\cZ{{\mathcal Z}}

\def\C{\mathbb{C}}
\def\F{\mathbb{F}}
\def\K{\mathbb{K}}
\def\Z{\mathbb{Z}}
\def\R{\mathbb{R}}
\def\Q{\mathbb{Q}}
\def\N{\mathbb{N}}
\def\M{\mathrm{M}}
\def\L{\mathbb{L}}
\def\M{{\normalfont\textsf{M}}} 
\def\U{\mathbb{U}}
\def\P{\mathbb{P}}
\def\A{\mathbb{A}}
\def\fp{\mathfrak{p}}
\def\fq{\mathfrak{q}}
\def\n{\mathfrak{n}}
\def\X{\mathcal{X}}
\def\x{\textrm{\bf x}}
\def\w{\textrm{\bf w}}
\def\ovQ{\overline{\Q}}
\def \Kab{\K^{\mathrm{ab}}}
\def \Qab{\Q^{\mathrm{ab}}}
\def \Qtr{\Q^{\mathrm{tr}}}
\def \Kc{\K^{\mathrm{c}}}
\def \Qc{\Q^{\mathrm{c}}}
\def\ZK{\Z_\K}
\def\ZKS{\Z_{\K,\cS}}
\def\ZKSf{\Z_{\K,\cS_{f}}}
\def\RSf{R_{\cS_{f}}}
\def\RTf{R_{\cT_{f}}}
\def \wH{{\mathrm H}}

\def\S{\mathcal{S}}
\def\vec#1{\mathbf{#1}}
\def\ov#1{{\overline{#1}}}
\def\sign{{\operatorname{sign}}}
\def\Vol{{\operatorname{Vol}}}
\def\Gm{\G_{\textup{m}}}
\def\fA{{\mathfrak A}}
\def\fB{{\mathfrak B}}

\def \GL{\mathrm{GL}}
\def \Mat{\mathrm{Mat}}

\def\house#1{{%
    \setbox0=\hbox{$#1$}
    \vrule height \dimexpr\ht0+1.4pt width .5pt depth \dp0\relax
    \vrule height \dimexpr\ht0+1.4pt width \dimexpr\wd0+2pt depth \dimexpr-\ht0-1pt\relax
    \llap{$#1$\kern1pt}
    \vrule height \dimexpr\ht0+1.4pt width .5pt depth \dp0\relax}}


\newenvironment{notation}[0]{%
  \begin{list}%
    {}%
    {\setlength{\itemindent}{0pt}
     \setlength{\labelwidth}{1\parindent}
     \setlength{\labelsep}{\parindent}
     \setlength{\leftmargin}{2\parindent}
     \setlength{\itemsep}{0pt}
     }%
   }%
  {\end{list}}

\newenvironment{parts}[0]{%
  \begin{list}{}%
    {\setlength{\itemindent}{0pt}
     \setlength{\labelwidth}{1.5\parindent}
     \setlength{\labelsep}{.5\parindent}
     \setlength{\leftmargin}{2\parindent}
     \setlength{\itemsep}{0pt}
     }%
   }%
  {\end{list}}
\newcommand{\Part}[1]{\item[\upshape#1]}

\def\Case#1#2{%
\smallskip\paragraph{\textbf{\boldmath Case #1: #2.}}\hfil\break\ignorespaces}

\def\Subcase#1#2{%
\smallskip\paragraph{\textit{\boldmath Subcase #1: #2.}}\hfil\break\ignorespaces}

\renewcommand{\a}{\alpha}
\renewcommand{\b}{\beta}
\newcommand{\g}{\gamma}
\newcommand{\bnu}{\bm{\nu}}
\newcommand{\bk}{\bm{k}}
\renewcommand{\d}{\Delta}
\newcommand{\e}{\epsilon}
\newcommand{\f}{\varphi}
\newcommand{\fhat}{\hat\varphi}
\newcommand{\bfphi}{{\boldsymbol{\f}}}
\renewcommand{\l}{\lambda}
\renewcommand{\k}{\kappa}
\newcommand{\lhat}{\hat\lambda}
\newcommand{\bfmu}{{\boldsymbol{\mu}}}
\renewcommand{\o}{\omega}
\renewcommand{\r}{\rho}
\newcommand{\rbar}{{\bar\rho}}
\newcommand{\s}{\sigma}
\newcommand{\sbar}{{\bar\sigma}}
\renewcommand{\t}{\tau}
\newcommand{\z}{\zeta}

\newcommand{\balpha}{\bm{\alpha}}
\newcommand{\bone}{\bm{1}}
\newcommand{\bv}{\bm{v}}
\newcommand{\bH}{\bm{H}}

\newcommand{\ga}{{\mathfrak{a}}}
\newcommand{\gb}{{\mathfrak{b}}}
\newcommand{\gn}{{\mathfrak{n}}}
\newcommand{\gp}{{\mathfrak{p}}}
\newcommand{\gP}{{\mathfrak{P}}}
\newcommand{\gq}{{\mathfrak{q}}}
\newcommand{\h}{{\mathfrak{h}}}
\newcommand{\Abar}{{\bar A}}
\newcommand{\Ebar}{{\bar E}}
\newcommand{\kbar}{{\bar k}}
\newcommand{\Kbar}{{\bar K}}
\newcommand{\Pbar}{{\bar P}}
\newcommand{\Sbar}{{\bar S}}
\newcommand{\Tbar}{{\bar T}}
\newcommand{\gbar}{{\bar\gamma}}
\newcommand{\lbar}{{\bar\lambda}}
\newcommand{\ybar}{{\bar y}}
\newcommand{\phibar}{{\bar\f}}

\newcommand{\Acal}{{\mathcal A}}
\newcommand{\Bcal}{{\mathcal B}}
\newcommand{\Ccal}{{\mathcal C}}
\newcommand{\Dcal}{{\mathcal D}}
\newcommand{\Ecal}{{\mathcal E}}
\newcommand{\Fcal}{{\mathcal F}}
\newcommand{\Gcal}{{\mathcal G}}
\newcommand{\Hcal}{{\mathcal H}}
\newcommand{\Ical}{{\mathcal I}}
\newcommand{\Jcal}{{\mathcal J}}
\newcommand{\Kcal}{{\mathcal K}}
\newcommand{\Lcal}{{\mathcal L}}
\newcommand{\Mcal}{{\mathcal M}}
\newcommand{\Ncal}{{\mathcal N}}
\newcommand{\Ocal}{{\mathcal O}}
\newcommand{\Pcal}{{\mathcal P}}
\newcommand{\Qcal}{{\mathcal Q}}
\newcommand{\Rcal}{{\mathcal R}}
\newcommand{\Scal}{{\mathcal S}}
\newcommand{\Tcal}{{\mathcal T}}
\newcommand{\Ucal}{{\mathcal U}}
\newcommand{\Vcal}{{\mathcal V}}
\newcommand{\Wcal}{{\mathcal W}}
\newcommand{\Xcal}{{\mathcal X}}
\newcommand{\Ycal}{{\mathcal Y}}
\newcommand{\Zcal}{{\mathcal Z}}

\renewcommand{\AA}{\mathbb{A}}
\newcommand{\BB}{\mathbb{B}}
\newcommand{\CC}{\mathbb{C}}
\newcommand{\FF}{\mathbb{F}}
\newcommand{\G}{\mathbb{G}}
\newcommand{\KK}{\mathbb{K}}
\newcommand{\NN}{\mathbb{N}}
\newcommand{\PP}{\mathbb{P}}
\newcommand{\QQ}{\mathbb{Q}}
\newcommand{\RR}{\mathbb{R}}
\newcommand{\ZZ}{\mathbb{Z}}

\newcommand{\bfa}{{\boldsymbol a}}
\newcommand{\bfb}{{\boldsymbol b}}
\newcommand{\bfc}{{\boldsymbol c}}
\newcommand{\bfd}{{\boldsymbol d}}
\newcommand{\bfe}{{\boldsymbol e}}
\newcommand{\bff}{{\boldsymbol f}}
\newcommand{\bfg}{{\boldsymbol g}}
\newcommand{\bfi}{{\boldsymbol i}}
\newcommand{\bfj}{{\boldsymbol j}}
\newcommand{\bfk}{{\boldsymbol k}}
\newcommand{\bfm}{{\boldsymbol m}}
\newcommand{\bfp}{{\boldsymbol p}}
\newcommand{\bfr}{{\boldsymbol r}}
\newcommand{\bfs}{{\boldsymbol s}}
\newcommand{\bft}{{\boldsymbol t}}
\newcommand{\bfu}{{\boldsymbol u}}
\newcommand{\bfv}{{\boldsymbol v}}
\newcommand{\bfw}{{\boldsymbol w}}
\newcommand{\bfx}{{\boldsymbol x}}
\newcommand{\bfy}{{\boldsymbol y}}
\newcommand{\bfz}{{\boldsymbol z}}
\newcommand{\bfA}{{\boldsymbol A}}
\newcommand{\bfF}{{\boldsymbol F}}
\newcommand{\bfB}{{\boldsymbol B}}
\newcommand{\bfD}{{\boldsymbol D}}
\newcommand{\bfG}{{\boldsymbol G}}
\newcommand{\bfI}{{\boldsymbol I}}
\newcommand{\bfM}{{\boldsymbol M}}
\newcommand{\bfP}{{\boldsymbol P}}
\newcommand{\bfX}{{\boldsymbol X}}
\newcommand{\bfY}{{\boldsymbol Y}}
\newcommand{\bfzero}{{\boldsymbol{0}}}
\newcommand{\bfone}{{\boldsymbol{1}}}

\newcommand{\aff}{{\textup{aff}}}
\newcommand{\Aut}{\operatorname{Aut}}
\newcommand{\Berk}{{\textup{Berk}}}
\newcommand{\Birat}{\operatorname{Birat}}
\newcommand{\characteristic}{\operatorname{char}}
\newcommand{\codim}{\operatorname{codim}}
\newcommand{\Crit}{\operatorname{Crit}}
\newcommand{\critwt}{\operatorname{critwt}} 
\newcommand{\cond}{\operatorname{cond}}
\newcommand{\Cycle}{\operatorname{Cycles}}
\newcommand{\diag}{\operatorname{diag}}
\newcommand{\Disc}{\operatorname{Disc}}
\newcommand{\Div}{\operatorname{Div}}
\newcommand{\Dom}{\operatorname{Dom}}
\newcommand{\End}{\operatorname{End}}
\newcommand{\ExtOrbit}{\mathcal{EO}} 
\newcommand{\Fbar}{{\bar{F}}}
\newcommand{\Fix}{\operatorname{Fix}}
\newcommand{\FOD}{\operatorname{FOD}}
\newcommand{\FOM}{\operatorname{FOM}}
\newcommand{\Gal}{\operatorname{Gal}}
\newcommand{\genus}{\operatorname{genus}}
\newcommand{\GITQuot}{/\!/}
\newcommand{\GR}{\operatorname{\mathcal{G\!R}}}
\newcommand{\Hom}{\operatorname{Hom}}
\newcommand{\Index}{\operatorname{Index}}
\newcommand{\Image}{\operatorname{Image}}
\newcommand{\Isom}{\operatorname{Isom}}
\newcommand{\hhat}{{\hat h}}
\newcommand{\Ker}{{\operatorname{ker}}}
\newcommand{\Ksep}{K^{\textup{sep}}}  
\newcommand{\lcm}{{\operatorname{lcm}}}
\newcommand{\LCM}{{\operatorname{LCM}}}
\newcommand{\Lift}{\operatorname{Lift}}
\newcommand{\limstar}{\lim\nolimits^*}
\newcommand{\limstarn}{\lim_{\hidewidth n\to\infty\hidewidth}{\!}^*{\,}}
\newcommand{\llog}{\log\log}
\newcommand{\logplus}{\log^{\scriptscriptstyle+}}
\newcommand{\maxplus}{\operatornamewithlimits{\textup{max}^{\scriptscriptstyle+}}}
\newcommand{\MOD}[1]{~(\textup{mod}~#1)}
\newcommand{\Mor}{\operatorname{Mor}}
\newcommand{\Moduli}{\mathcal{M}}
\newcommand{\Norm}{{\operatorname{\mathsf{N}}}}
\newcommand{\notdivide}{\nmid}
\newcommand{\normalsubgroup}{\triangleleft}
\newcommand{\NS}{\operatorname{NS}}
\newcommand{\onto}{\twoheadrightarrow}
\newcommand{\ord}{\operatorname{ord}}
\newcommand{\Orbit}{\mathcal{O}}
\newcommand{\Per}{\operatorname{Per}}
\newcommand{\Perp}{\operatorname{Perp}}
\newcommand{\PrePer}{\operatorname{PrePer}}
\newcommand{\PGL}{\operatorname{PGL}}
\newcommand{\Pic}{\operatorname{Pic}}
\newcommand{\Prob}{\operatorname{Prob}}
\newcommand{\Proj}{\operatorname{Proj}}
\newcommand{\Qbar}{{\bar{\QQ}}}
\newcommand{\rank}{\operatorname{rank}}
\newcommand{\rad}{\operatorname{rad}}
\newcommand{\Rat}{\operatorname{Rat}}
\newcommand{\Res}{{\operatorname{Res}}}
\newcommand{\Resultant}{\operatorname{Res}}
\renewcommand{\setminus}{\smallsetminus}
\newcommand{\sgn}{\operatorname{sgn}}
\newcommand{\SL}{\operatorname{SL}}
\newcommand{\Span}{\operatorname{Span}}
\newcommand{\Spec}{\operatorname{Spec}}
\renewcommand{\ss}{{\textup{ss}}}
\newcommand{\stab}{{\textup{stab}}}
\newcommand{\Stab}{\operatorname{Stab}}
\newcommand{\Support}{\operatorname{Supp}}
\newcommand{\Sym}{\operatorname{Sym}}  
\newcommand{\tors}{{\textup{tors}}}
\newcommand{\Trace}{\operatorname{Trace}}
\newcommand{\trianglebin}{\mathbin{\triangle}} 
\newcommand{\tr}{{\textup{tr}}} 
\newcommand{\UHP}{{\mathfrak{h}}}    
\newcommand{\Wander}{\operatorname{Wander}}
\newcommand{\<}{\langle}
\renewcommand{\>}{\rangle}

\newcommand{\pmodintext}[1]{~\textup{(mod}~#1\textup{)}}
\newcommand{\ds}{\displaystyle}
\newcommand{\longhookrightarrow}{\lhook\joinrel\longrightarrow}
\newcommand{\longonto}{\relbar\joinrel\twoheadrightarrow}
\newcommand{\SmallMatrix}[1]{%
  \left(\begin{smallmatrix} #1 \end{smallmatrix}\right)}
  
  \def\({\left(}
\def\){\right)}

\hypersetup{breaklinks=true}


\title[Multiplicatively dependent vectors on a hyperplane]
{Multiplicatively dependent integer vectors on a hyperplane}
\author[M. Afifurrahman, V. Iverson, G. C. Sanjaya]{Muhammad Afifurrahman, Valentio Iverson, Gian Cordana Sanjaya}

\address[MA]{School of Mathematics and Statistics, University of New South Wales, Sydney NSW 2052, Australia}
\email{m.afifurrahman@unsw.edu.au}
\address[VI]{Department of Pure Mathematics, University of Waterloo, Waterloo, ON, Canada}
\email{viverson@uwaterloo.ca}
\address[GCS]{Department of Pure Mathematics, University of Waterloo, Waterloo, ON, Canada}
\email{gcsanjaya@uwaterloo.ca}

\thanks{}
\begin{abstract}
We establish several asymptotic formulae and upper bounds for the count of multiplicatively dependent integer vectors that lie on a fixed hyperplane and have bounded height. This work constitutes a direct extension of the results obtained by Pappalardi, Sha, Shparlinski, and Stewart.
\end{abstract}

\maketitle

\tableofcontents

\section{Introduction}\label{sec:intro}
\subsection{Multiplicatively dependent vectors}

Let $G$ be a multiplicative group and $n$ be a positive integer. A  vector $\bm{\nu}=(\nu_1,\dotsc,\nu_n)\in G^n$ is \textit{multiplicatively dependent} if there exists a nonzero vector $\bk=(k_1,\dotsc,k_n)\in \Z^n$ with \begin{align}\label{eqn:multdep}
    \nu_1^{k_1}\dotsc \nu_n^{k_n}=1.
\end{align}
We denote $\cM_n(G)$ as the set of multiplicatively dependent vectors in $G^n$.

We restrict our attention to the set of multiplicatively dependent vectors of algebraic numbers, $\cM_n(\overline{\Q}^*)$. There have been many studies in this set, in particular regarding the height of the vector $\bnu$ and the exponent vector $\bk$. With respect to the exponent vector, Loxton and van der Poorten~\cite{LvdP,vdPL}, Matveev~\cite{Matveev}, and Loher and Masser~\cite[Corollary~3.2]{LM} (attributed to K.~Yu) proved that for a fixed multiplicatively dependent vector $\bnu$, there exists an exponent vector $\bk$ of small height that satisfies~\eqref{eqn:multdep}. The exact statement of this result is reproduced in Lemma~\ref{lem:exp}, due to~\cite[Theorem~1]{vdPL}.
In addition, Stewart~\cite[Theorem~1]{Stewart} gave an inequality for the heights of the coordinates of such $\bnu$ with a low multiplicative rank. A lower bound for the sum of the heights of the coordinates is implied in Vaaler~\cite{Vaaler}.

The central motivation for this work lies in the study of arithmetic statistics of multiplicatively dependent vectors of algebraic integers. This was influenced by the work of Pappalardi, Sha, Shparlinski, and Stewart \cite{PSSS}, who studied such statistics for vectors whose coordinates are algebraic integers of bounded height within a fixed degree or within a fixed number field. Since then, the distribution of multiplicatively related vectors in a given base field has been extensively studied, including work on their density \cite{KSSS,Stew19} and translations \cite{DS2018}.

We also observe that this problem is equivalent to the problem of counting points contained within some proper algebraic subgroup of $\G^n_m$, where $\G_m$ is the multiplicative group of an algebraic closure of $\Q$, see~\cite{Schmidt96}. In addition, we note that multiplicatively independent vectors plays an important role in dynamical Mordell-Lang conjecture, as seen at~\cite{OS15,SV}.

In another direction, there have been various research directions concerning multiplicatively dependent vectors that lie on a variety, specifically a curve. For example,~\cite{OSSZ} studies abelian multiplicatively dependent points on a curve in a torus. Specialising to elliptic curves, such result may be seen at~\cite{BCMOS,BS20,BBHOS}. Another problem of this flavor, as seen at~\cite{BBGMOS,Youn24}, addresses multiplicative dependence among the values of several fixed polynomials.

Our work combines these directions as a natural extension of~\cite{PSSS}. Here we discuss the statistics of multiplicatively dependent $n$-dimensional integer vectors on a variety, specifically a hyperplane. As per our knowledge, such result was done only for $n=2$, as seen at~\cite{ BMZ99, DS2018}. In addition,~\cite{BMZ08} also considers a parallel problem on the finiteness of the number of such vectors where the variety is a ``deprived" plane $\cP^{oa}$ of dimension $2$.

Formally, for a fixed positive integer  $n$ and $J\in \C$ and $\balpha \in \Z^n$, we count the number of vectors $\bnu$ such that the following conditions hold:
\begin{itemize}
    \item the coordinates $\nu_1,\dots,\nu_n$ are in a set $A$, where $A$ is either $\Z$ or $\Z^+$ (the set of positive integers),
    \item for all $1\le i\le n$, the inequality $0<|\nu_i|\le H$ holds,
    \item the vector satisfies $\alpha_1\nu_1+\dots+\alpha_n\nu_n=J$ or $\balpha \cdot \bnu =J$ and
\item it is multiplicatively dependent.
\end{itemize}
Our goal is to derive results that hold uniformly, as $|J|$ and $H$ goes to infinity.

We first recall that~\cite[Theorem~1.4]{PSSS} implies the number of such vectors in $\Z^n$ is of order $H^{n-1}$. Separately, standard techniques from geometry of numbers (see Section~\ref{sec:prelim-lat}) show that the number of vectors in $\Z^n$ lying on a fixed hyperplane is also of order $H^{n-1}$. In the next section, we demonstrate that these two conditions are usually independent; except for some classes of hyperplanes, the number of of such vectors that are multiplicatively dependent and lie on a fixed hyperplane is of order $H^{n-2}$.

We remark that the present work only considers integer vectors due to constraints arising from some of our tools in the arguments. Neverthless, our main argument may be extended to algebraic integers, analogous to the setting at~\cite{PSSS}. We refer to Section~\ref{sec:further-ext} for further disscussions on this direction.

\subsection{Counting vectors on a hyperplane}\label{sec:results}
Define
\begin{align*}
    \cS_n(H,J; \balpha)\coloneqq \left\{\bnu \in \cM_n(\C^*) \cap [-H,H]^n \colon \balpha \cdot \bnu = J\right\}
\end{align*}
and let $S_n(H,J;\balpha)=\#\cS_n(H,J;\balpha)$. In addition, let $k$ be the number of nonzero coordinates of $\balpha$. Previous result from~\cite[Theorem 1.4]{PSSS}, restated in~\eqref{eqn:PSSS}, considered the case $k=0$, which corresponds to counting multiplicatively dependent vectors $\bnu \in \Z^n$ in the box $[-H,H]^n$. 

We are ready to state our main result in this section related to $S_n(H,J;\balpha)$ for $k\ge 3$.
\begin{theorem}\label{thm:main-k>=3}
Let $n \ge 3$ be a positive integer.
Let $J$ be an integer and $\balpha \in \Z^n$ be a vector with $k\ge 3$ nonzero coordinates.
Then, there exists an explicitly computable constant $C_{\balpha;J} \ge 0$ that depends on $\balpha$ and $J$ such that for any real number $H \gg |J|$,
\[ S_n(H,J;\balpha) = C_{\balpha;J} H^{n-2} + \begin{cases}
    O (H^{n-5/2} + |J|^2 H^{n - 4}) & \quad \text{if} \quad k\ge 5,\\
    O(H^{n-5/2}(\log H)^{16} + |J| H^{n - 3}) & \quad \text{if} \quad k=3,4 \text{ and }J\ne 0.
\end{cases} \]
\end{theorem}
A formula for $C_{\balpha;J}$ is given in~\eqref{eqn:Calpha} and~\eqref{eqn:Calpha-k=3} for $k\ge 4$ and $k=3$, respectively. The constant $C_{\balpha;J}$ is generally nonzero, except for certain choices of $\balpha$ and $J$ which are discussed in Section~\ref{sec:strength}.

For completeness, we mention a corollary of this result where we fix $J$ (which also fix the hyperplane) and let $H\to \infty$.
\begin{corollary}\label{cor:k>=3-fixed}
    If we fix $J$ and let $H\to \infty$, we obtain \[ S_n(H,J;\balpha) = C_{\balpha;J} H^{n-2} + \begin{cases}
    O (H^{n-5/2}) & \quad \text{if} \quad k\ge 5,\\
    O(H^{n-5/2}(\log H)^{16}) & \quad \text{if} \quad k=3,4 \text{ and } J\ne 0.
\end{cases}  \]
\end{corollary}

We now discuss our results for the remaining case, $k=1,2$. When $k=2$, we obtain a formula in a similar setup from Corollary~\ref{cor:k>=3-fixed}.
\begin{theorem}\label{thm:main-k=2}
Let $J \neq 0$ be an integer and $\balpha \in \Z^n$ with exactly two nonzero elements. Then, there exists a constant $C_{\balpha;J} \ge 0$ that depends on $\balpha$ and $J$ such that as $H \to \infty$,
\[ S_n(H,J;\balpha) = C_{\balpha;J} H^{n-2} + O (H^{n-5/2}(\log H)^{16}). \]
\end{theorem}
We do not have an equivalent result of Theorem~\ref{thm:main-k>=3} for $k=2$, due to reasons explained in Section~\ref{sec:strength}.

When $k=1$, the problem of counting $S_n(H,J;\balpha)$ is equivalent to counting the number of multiplicatively dependent vectors $\bnu \in \Z^n$ in the box $[-H,H]^n$ such that one of its coordinates are fixed, say $\alpha_1=J/\|\balpha\|$. Therefore, we may let $\balpha=\bfe_1 = (1, 0, 0, \ldots, 0)$ as the first standard basis vector in $\Z^n$ and $J$ be fixed. In this setup, we obtain the following result.
\begin{theorem}\label{thm:main-k=1}
Let $n \geq 3$ and $J \ne 0$ be an integer. Then, there exists real constants $C^{(0)}_J$, $C^{(1)}_J$ that depend on $J$ with
\[ S_n(H,J;\bfe_1) = \begin{cases} 
    (2H)^{n-1} & \quad \text{if} \quad |J| = 1, \\
    C^{(1)}_JH^{n-2}\left\lfloor \dfrac{\log H}{\log f(|J|)}\right\rfloor+C^{(0)}_JH^{n-2}+O(H^{n-5/2}) & \quad \text{if} \quad |J| > 1,
\end{cases} \]
    where for any positive integer $A > 1$, $f(A)$ is the smallest positive integer $B$ such that $A$ is a power of $B$.
\end{theorem}
We provide the exact definition of $C^{(1)}_J$ and $C^{(0)}_J$ in~\eqref{eqn:Cbfe1}. Note that the different result for $|J| = 1$ is expected, as every vector whose coordinate contains $\pm 1$  is multiplicatively dependent.

These results can be compared with Pappalardi, Sha, Shparlinski and Stewart's result~\cite[Theorem~1.4]{PSSS} as follows.
Denote $S_n(H)$ and $S_n^*(H)$ as the number of multiplicatively dependent vectors $\bnu \in \Z^n$  of height at most $H$. 
By substituting $d=1$ in~\cite[Equation~(1.17)]{PSSS}, we obtain 
\begin{equation}\label{eqn:PSSS}\begin{split}
    S_n(H)=2^{n-1}n(n+1)H^{n-1}+\begin{cases}
        O(\log H) & \quad \text{if} \quad n = 2, \\
        O(H^{n-2}\exp(c_n\log H/\log \log H)) & \quad \text{if} \quad n\ge 3,
    \end{cases}
    \end{split}
\end{equation}
where $c_n > 0$ only depends on $n$. Note that we give a different error term for $S_2(H)$ compared to~\cite[Equation (1.17)]{PSSS}, due to reasons explained in Section~\ref{sec:cor-PSSS}.

A variant of our main problem is to count multiplicatively dependent positive integer points.
For fixed positive integers $J$ and $n>1$, we denote $S^+_n(H,J;\balpha)$ as the number of multiplicatively dependent positive integer vectors $(\nu_1,\dots,\nu_n)$ that also satisfies $\balpha \cdot \bnu = J$. 
In comparison, if we do not require the hyperplane condition, we may follow~\cite{PSSS} to obtain the number of multiplicatively dependent positive integer vectors of height at most $H$ is $S_n(H)/2^n$, where $S_n(H)$ is defined before~\eqref{eqn:PSSS}.

We obtain the following result with a few caveats on the condition for $\balpha$.

\begin{theorem}\label{thm:main-positive-integer}
With the same requirements as Theorem~\ref{thm:main-k>=3}, the same results also hold if we replace $S_n$ with $S_n^+$ and replacing $C_{\balpha;J}$ with another constant $C^+_{\balpha;J}$ for $k\ge 4$, as long as $\balpha$ has at least two positive and two negative coordinates, with the same error terms.
  
In addition, for $k=1$, we obtain
\[ S_n^+(H,J;\bfe_1) = \begin{cases} 
    H^{n-1} & \quad \text{if} \quad J = 1, \\
    \dfrac{C^{(1)}_J}{2^{n-1}}H^{n-2}\left\lfloor \dfrac{\log H}{\log f(|J|)}\right\rfloor+\dfrac{C^{(0)}_J}{2^{n-1}}H^{n-2}+O(H^{n-5/2}) & \quad \text{if} \quad J > 1,
\end{cases} \]
where $C^{(1)}_J$ and $C^{(0)}_J$ are defined in~\eqref{eqn:Cbfe1} and $f(A)$ is defined in Theorem~\ref{thm:main-k=1}.
\end{theorem}
In addition, we also discuss the case where $\balpha$ is a nonnegative integer vector. In this case, we automatically have $J\ll H$. Therefore, in this case we instead denote the related quantity as $S^+_n(J;\balpha)$ and obtain our asymptotical formula in the following form.

\begin{theorem}\label{thm:main-all-positive-integer}
If all coordinates of $\balpha$ are positive,  there exists a computable constant $C^+_{\balpha,J}\ge 0$, derived before Proposition~\ref{prp:allpos}, such that     
\[ S_n^+(J;\balpha)= C^+_{\balpha,J} J^{n-2} +\begin{cases}
    O(J^{n-5/2}) &\quad\text{if}\quad n\ge 5, \\
    O(J^{n-5/2}(\log J)^{16}) &\quad\text{if} \quad 3\le n \le 4.
\end{cases} \] 
\end{theorem}

A special case of Theorem~\ref{thm:main-all-positive-integer} where $\balpha=\bone_n$ is discussed in Section~\ref{sec:allone}, with alternative proofs and possible extensions. Also, we note that our methods of proof may be modified for the omitted cases of $\balpha$, as discussed in Section~\ref{sec:positive}.

To close the discussions of these results, note that the error terms $O(H^{n-5/2+o(1)})$ is expected, as discussed in Section~\ref{sec:strength}.

\subsection{Structure of paper and outline of arguments}
Section \ref{sec:intro} introduces our problems, provides necessary background, and states our main results. This section also clarifies the conventions and notations used throughout this work.

Section \ref{sec:prelim-arit} collects several classical arithmetical results essential for our arguments.
Section~\ref{sec:bompil} derives some bounds on counting integer solutions in a box $[-H,H]$ to a system of equations of the form (for example) \begin{align*}
    \nu_1^{k_1}\nu_{2}^{k_2}=\nu_3^{k_3}, \quad
    \alpha_1\nu_1+\alpha_2\nu_2+\alpha_3\nu_3=J.
\end{align*} The main tool of our arguments of this section, which can be seen at Section~\ref{sec:proof-bompil}, is based on a quantitative version~\cite{CCDN} of Bombieri-Pila's~\cite{BP} celebrated determinant method for counting integer points on a curve.

Section~\ref{sec:prelim-lat} derives results that are related to our volume-counting arguments, based on related results from Davenport~\cite{Dav} and Marichal-Mossinghoff~\cite{MarMos}.

Section~\ref{sec:upper-bounds} considers the problem of counting vectors in $\cS(H,J;\balpha)$ of a given rank $r$. Following similar arguments as in~\cite{PSSS} with additional power saving that comes from the hyperplane equation $\balpha \cdot \bnu = J$, we obtain a similar upper bound as in~\cite[Proposition~1.5]{PSSS} with some additional conditions related to some inequalities between $k$ and $r$. In particular, our result in this section is sufficient for our use case when $k\ge 5$. For $k<5$, we improve this upper bound in Section~\ref{sec:too-many-zeroes} where we instead use results from Section~\ref{sec:bompil} for $2\le k\le 4$ and obtain a new result for $k=1$.

Section~\ref{sec:small-rank} deals with deriving an asymptotic formula of the number of such vectors, which completes the proof of our main results. Section~\ref{sec:positive} considers similar problems where we instead replace the box $[-H,H]^n$ with the positive integer box $[1,H]^n$. Most arguments can be easily translated to this setup, except for the case where $\balpha$'s coordinates are all positive.

Section~\ref{sec:remark} is dedicated to various remarks related to our results. First, Section~\ref{sec:strength} discusses the strength of our results with respect to the variable dependence and the order of the error terms. Section~\ref{sec:cond-bompil} discusses the restrictions placed on the results of Section~\ref{sec:bompil}. Next, Section~\ref{sec:cor-PSSS} provides a correction to the error term of a formula of~\cite{PSSS}. Section~\ref{sec:allone} discusses an alternate proof of a result in Section~\ref{sec:positive} in the case where $\balpha=\bone_n$ and a related problem of counting multiplicatively dependent partitions.

Finally, Section~\ref{sec:remark-further} is devoted to discussing several related problems and background of our work. In particular, Section~\ref{sec:small-dim} discuss various extensions for $n=2,3$ and Section~\ref{sec:further-ext} discuss some possible further extensions, in particular the constraints of extending our arguments to algebraic integers.
 
We now discuss the outline of our main arguments for counting $S_n(H,J;\balpha)$. We first count the number of such vectors of a given multiplicative rank $r$ in Section~\ref{sec:upper-bounds}, which is done  in Corollary~\ref{cor:Snrupper}. The result depends on $k$, the number of nonzero coefficient of $\balpha$, and is sufficient for our purposes when $k\ge 5$.

For the case $k\le 4$, we instead appeal to our results in Section~\ref{sec:bompil} on bounding the number of solutions to a system of equations based on the multiplicative relations and hyperplane equations. These bounds are based on applying a variant of Bombieri-Pila's determinant method for counting integer points in a certain curve, as seen in Proposition~\ref{prp:CCDN}.

These arguments imply that for most choices of $\balpha $ and $J$, the main term of our asymptotic formula for $S_{n}(H,J;\balpha)$ comes from counting vectors of rank $0$ and $1$. In these cases, we may classify such vectors and convert our problem by using Davenport's~\cite{Dav} celebrated lemma to some volume-counting problems over $\R^{n-1}$. As seen in~\eqref{eqn:Calpha}, our counting problem is equivalent to counting volumes of some regions obtained by intersecting a box and a hyperplane related to the vector $\balpha$. We are practically done here, however if one wants to obtain an explicit coefficient one may use Marichal-Mossinghoff's formula~\cite{MarMos} (see Section~\ref{sec:count-vol}) to obtain such result.

\subsection{Conventions and notations}

We always implicitly assume that $H$ is large enough, in particular so that the logarithmic expression $\log H$ and $\log \log H$ are well defined.

Following~\cite{PSSS}, we only consider vectors with nonzero coordinates. However, our arguments can easily be extended to vectors with zero coordinates, except probably when defining the corresponding multiplicative rank (by following the convention $0^0=1$). 

We use the notation $\Vol_k$ for the volume of a $k$-dimensional polytope inside a $k$-dimensional affine subspace of $\R^n$ in the sense of Lebesgue measure.

We use the Landau symbols $O$ and $o$ and the Vinogradov symbol $\ll$. We recall that the assertions $U=O(V)$ and $U\ll V$  are both equivalent to the inequality $|U|\le cV$ with some positive constant $c$, while $U=o(V)$ means that $U/V\to 0$. Unless stated otherwise, all constants of the error terms in this work  are dependent on the related hyperplane vector $\balpha$.

For any real number $x$, let $\lceil x\rceil$ denote the smallest integer greater than or equal to $x$, and let $\lfloor x\rfloor$ denote the greatest integer less than or equal to $x$. Also, denote $\rad x$ as the product of all prime factors of $x$.

For a finite set $S$ we use $\#{S}$ to denote its cardinality. Also, we always denote vector as boldfaced. In addition, for $\balpha \in \Z^n$, we denote
$\gcd(\balpha) := \gcd(\alpha_1, \alpha_2, \ldots, \alpha_n)$.

\section{Preliminary results}
\subsection{Arithmetic results}\label{sec:prelim-arit}
We first recall a lemma of van der Poorten and Loxton~\cite[Theorem~1]{vdPL} on the exponents related to a multiplicatively dependent vector.

\begin{lemma}\label{lem:exp}
Let $n\ge 2$, and let $\alpha_1,\:\dots,\:\alpha_n$ be multiplicatively dependent nonzero integers of height at most $H$. Then, there is a positive constant $c_n$ which depends only on $n$ and a nonzero vector $\bk\in \Z^n$ with 
\[ \max_{1\le i \le n} |k_i|<c_n(\log H)^{n-1} \]
    such that
\[ \alpha_1^{k_1}\dots\alpha_n^{k_n}=1. \]
\end{lemma}

Let $x$ and $y>2$ be positive real numbers, and let $\psi(x,y)$ denote the number of positive integers not exceeding $x$ which have no prime factors greater than $y$. Define
\[ Z\coloneqq\left(\log \left(1+\dfrac{y}{\log x}\right)\right)\dfrac{\log x}{\log y}+\left(\log \left(1+\dfrac{\log x}{y}\right)\right)\dfrac{y}{\log y} \]
    and
\[ u \coloneqq (\log x)/(\log y). \]
We have the following asymptotic on $\psi(x,y)$, taken from~\cite{Bru}.

\begin{lemma}\cite[Theorem 1]{Bru}\label{lem:divdeBruijn}
    For $2<y\le x$, we have
    \[ \psi(x,y)=\exp\left(Z\left(1 +O((\log y)^{-1})+O((\log \log x)^{-1})+O((u+1)^{-1})\right)\right). \]
\end{lemma}

We use Lemma~\ref{lem:divdeBruijn} in the following form, where we count positive integers whose prime factors divide a given integer.

\begin{lemma}\label{lem:divdeBruijn-div}
Let $x$ and $y > 2$ be positive integers, and let $\psi_0(x, y)$ denote the number of positive integers not exceeding $x$ whose prime factors divide $y$. Then there exists an absolute constant $c$ such that
\[ \psi_0(x, y) < \exp(c \log \max\{x, y\}/\log \log y). \]
\end{lemma}
\begin{proof}
First, by a classical result due to~\cite{HW}, there are at most $c_0 \log y/\log \log y$ distinct prime numbers dividing $y$, where $c_0$ is an absolute constant (see also~\cite[Th\'eor\`eme 11]{Rob} for an explicit value of $c_0$). Thus we have $\psi_0(x, y) \leq \psi(x, p)$, where $p$ is the $\lceil c_0 \log y/\log \log y \rceil$-th smallest prime. By the Prime Number Theorem, there exists an absolute constant $c_1$ such that $p \leq c_1 \log y$. Hence by Lemma~\ref{lem:divdeBruijn}, we have
\[ \psi_0(x, y) \leq \psi(x, p) \leq \psi(x, c_1 \log y) < \exp(c \log \max\{x, y\}/\log \log y) \]
    for some absolute constant $c$, as desired.
\end{proof}

To describe the case $k=2$ also need the following result on the finiteness of a Pell-like equation related to multiplicatively dependent pairs of integers. This result generalises~\cite[Section~4]{DS2018}, which considers the case $(\alpha,\beta)=(1,1)$; see also \cite{BL2006,SS} for related equations and Section~\ref{sec:remark-further} for related discussions. Note that this lemma is a special case of \cite[Theorem 1]{BMZ99}.

\begin{lemma}\label{lem:extra-count-k=2}
For any nonzero integers $\alpha, \beta, J$, there exists finitely many multiplicatively dependent pairs $(x, y)$ of integers with $\alpha x + \beta y = J$.
\end{lemma}    
\begin{proof}
By considering all possible signs of $x$ and $y$, it suffices to count pairs $(x, y)$ with $x, y > 0$. Furthermore, we may also assume that $x, y > 1$. We further assume that $\gcd(\alpha, \beta) = 1$.

A multiplicatively dependent pair $(x, y)$ of integers with $x, y > 1$ takes the form $(z^k, z^m)$ for some positive integers $z > 1$, $k$, and $m$. Clearly, $\alpha z^k + \beta z^m = J$ implies $z \mid J$, so we have finitely many choices for $z$.

Now fix $z$. We have $z^{\min\{k, m\}} \mid J$, so there are only finitely many choices for $\min\{k, m\}$. However, fixing $k$ uniquely determines $m$, and fixing $m$ uniquely determines $k$, so there are finitely many choices for the pair $(k, m)$ for fixed $z$. This proves Lemma~\ref{lem:extra-count-k=2}.
\end{proof}

\subsection{Counting points on systems of equations}\label{sec:bompil} 

As mentioned previously, one of our main arguments on this paper depends on obtaining some bounds on the number of solutions to a system of equations that consists of the hyperplane and multiplicative dependence equations. When this system has a small number of variables, we use an argument based on Bombieri-Pila's~\cite{BP} determinant method for counting integer points of a curve in a box. In particular, we use the following explicit version from~\cite{CCDN}.

\begin{lemma}\label{lem:bompil-CCDN}\cite[Theorem~3]{CCDN}
Let $f \subseteq \A_{\Q}^2$ be an integral affine curve of degree $d$. Then there exists an absolute constant $c > 0$ such that for all $H \ge 1$, the number of integer points in $f$ with naive height at most $H$ is at most $cd^3 H^{1/d}(\log H+d)$.
\end{lemma}

Lemma~\ref{lem:bompil-CCDN} will be used in the following form that is sufficient for our purposes.

\begin{proposition}\label{prp:CCDN}
Let $f \subseteq \A_{\Q}^2$ be an affine curve of degree $d$ such that any linear factor of $f$ does not have any integer points. Then there exist absolute constants $C_0, C_1 > 0$ such that for all $H \ge 1$, the number of integer points in $f$ with naive height at most $H$ is at most
\[ C_0 dH^{1/2}(\log H + 2) + C_1 d^3 H^{1/3} (\log H + d). \]
\end{proposition}
\begin{proof}
Let $f=g_1\dots g_w$ be a factorisation of $f$ as irreducible curves. Note that each integer points of $f$ must be on one of $g_1, \ldots, g_w$. From our assumption, we may see that the curves $g_i$ with $\deg g_i=1$ admit no integer points. Then, when $\deg g_i \geq 2$,  we use Lemma~\ref{lem:bompil-CCDN} to conclude that this curve has at most
\[ c (\deg g_i)^3 H^{1/\deg g_i}(\log H + \deg g_i) \]
integer points of height at most $H$. Thus the number of integer points on $f$ with height at most $H$ is at most 
\begin{align*}
    & \sum_{i, \deg g_i \geq 2} c (\deg g_i)^3 H^{1/\deg g_i}(\log H + \deg g_i) \\
    \leq\;& \#\{i : \deg g_i = 2\} \cdot 8c H^{1/2}(\log H + 2) + \sum_{i, \deg g_i \geq 3} c (\deg g_i)^3 H^{1/3}(\log H + d) \\
    \leq\;& 4cd H^{1/2}(\log H + 2) + cd^3 H^{1/3}(\log H + d).
\end{align*}
Picking $C_0 = 4c$ and $C_1 = c$ completes our argument.
\end{proof}

We are now ready to state our main results of this section.
\begin{theorem}\label{thm:bompil-2var}
Let $A$, $B$, $\alpha_1$, $\alpha_2$, $J$  be nonzero integers and $k_1$, $k_2$, $k_3$ be positive integers. Then, the number of integer solutions $(\nu_1,\nu_2,\nu_3)$ to the system of equations
\begin{align}\label{eqn:fatal2var-1}
    &A\nu_1^{k_1} \nu_2^{k_2}= B\nu_3^{k_3}, \\
    &\alpha_1\nu_1+\alpha_2\nu_2=J, \label{eqn:fatal2var-linear}
\end{align}
    such that $0<|\nu_i|\le H$ for $i=1,2,3$ is bounded above  by $$C_{2}(k_1+k_2+k_3)H^{1/2} (\log H+2)+C_3(k_1+k_2+k_3)^3H^{1/3} (\log H+k_1+k_2+k_3)$$ for some absolute constants $C_2,C_3>0$.
    
In addition, the same result also holds if we replace~\eqref{eqn:fatal2var-1} with \begin{equation}\label{eqn:fatal2var-2}
    A \nu_1^{k_1} \nu_3^{k_3} = B\nu_2^{k_2}.
\end{equation}
\end{theorem}

\begin{theorem}\label{thm:bompil-3var}
    Let $A$, $B$, $\alpha_1$, $\alpha_2$, $\alpha_3$, $J$ be nonzero integers and $k_1$, $k_2$, $k_3$ be positive integers. Then, the number of integer solutions $(\nu_1,\nu_2,\nu_3)$ to the system of equations \begin{align}\label{eqn:fatal1}
        &A\nu_1^{k_1} \nu_2^{k_2}= B\nu_3^{k_3},\\
        &\alpha_1\nu_1+\alpha_2\nu_2+\alpha_3\nu_3=J, \label{eqn:fatal2}
    \end{align}
    such that $\alpha_1\nu_1,\alpha_2\nu_2 \ne J$ and $0<|\nu_i|\le H$  for $i=1,2,3$ is bounded above  by $$C_{2}(k_1+k_2+k_3)H^{1/2} (\log H+2)+C_3(k_1+k_2+k_3)^3H^{1/3} (\log H+k_1+k_2+k_3)$$ for some absolute constants $C_2,C_3>0$.
\end{theorem}

\begin{theorem}\label{thm:bompil-4var}
Let $J$ be a nonzero integer and $\balpha \in \Z^4$ such that at most one of $\alpha_1$, $\alpha_2$, $\alpha_3$ and $\alpha_4$ are zero, and $k_1$, $k_2$, $k_3$, $k_4$ be positive integers.
Then the number of integer solutions $(\nu_1,\nu_2,\nu_3,\nu_4)$ to the system of equations
\begin{equation}\label{eqn:fatal-4var}\begin{split}
    &\nu_1^{k_1} \nu_2^{k_2}= \nu_3^{k_3}\nu_4^{k_4},\\
    &\alpha_1\nu_1+\alpha_2\nu_2+\alpha_3\nu_3+\alpha_4\nu_4=J,
\end{split}
\end{equation}
    such that $0<|\nu_i|\le H$  for $i=1,2,3,4$ is bounded above  by $$C_{2}(k_1+k_2+k_3+k_4)H^{3/2} (\log H+2)+C_3(k_1+k_2+k_3+k_4)^3H^{4/3} (\log H+k_1+k_2+k_3+k_4)$$ for some absolute constants $C_2,C_3>0$.
\end{theorem}
The proofs of these results are derived in Section~\ref{sec:2<=k<=4}. In addition, we discuss the condition of these results in Section~\ref{sec:cond-bompil}. 

\subsection{Proofs of Theorems~\ref{thm:bompil-2var}-\ref{thm:bompil-4var}}\label{sec:proof-bompil}

\subsubsection{Proof of Theorem~\ref{thm:bompil-2var}} 

Substituting \eqref{eqn:fatal2var-linear} to~\eqref{eqn:fatal2var-1}, we obtain
\[ A\nu_1^{k_1}(J-\alpha_1\nu_1)^{k_2}=B\alpha_2^{k_2}\nu_3^{k_3}. \]
Then, we consider a curve $f_1\in \Z[x,y]$ defined as
\[ f_1(x,y) \coloneqq B\alpha_2^{k_2}y^{k_3}-Ax^{k_1}(J-\alpha_1x)^{k_2}. \]
Suppose that $f_1$ has a linear factor. We may consider factors of the form $y-Dx-E$, $x-D$ or $y-D$ for some rational numbers $D\ne 0$ and $E$; note that $x$ and $y$ are not factors of $f_1$, since $f_1(0, 0) \neq 0$. First, suppose that $x-D \mid f_1$. Then, 
\[ B\alpha_2^{k_2}y^{k_3}=AD^{k_1}(J-\alpha_1D)^{k_2} \]
holds true for all $y$, which is impossible since $B,\alpha_2,k_3\ne 0$. With similar arguments, we also have $y-D\nmid f_1$.

The remaining case is $y-Dx-E \mid f_1$, where here we conclude
\[ B\alpha_2^{k_2}(Dx+E)^{k_3}=Ax^{k_1}(J-\alpha_1x)^{k_2} \]
for all $x$. Comparing the constant coefficients, since $B,\alpha_2, k_1\ne 0$, we obtain $E=0$. Then the left hand side is a monomial, but the right hand side is not a monomial, since $A, J, \alpha_1, k_2 \neq 0$; contradiction.

Therefore, any irreducible factors of $f_1$ is at least of degree $2$. Applying Proposition~\ref{prp:CCDN} completes the proof of Theorem~\ref{thm:bompil-2var} in this case.

When~\eqref{eqn:fatal2var-1} is replaced with~\eqref{eqn:fatal2var-2}, we  instead use the curve \[
f_2(x,y)\coloneqq A\alpha_2^{k_2}x^{k_1}y^{k_3}-B(J-\alpha_1x)^{k_2}
\] and apply similar arguments as in the previous case. This completes the proof of Theorem~\ref{thm:bompil-2var}.

\subsubsection{Proof of Theorem~\ref{thm:bompil-3var}}

Substituting \eqref{eqn:fatal2} to \eqref{eqn:fatal1} with respect to $\nu_3$, we obtain
\[ A\alpha_3^{k_3}\nu_1^{k_1}\nu_2^{k_2} = B(J-\alpha_1\nu_1-\alpha_2\nu_2)^{k_3}. \]
Then, we consider a curve $f_3\in \Z[x,y]$ defined as 
\[ f_3(x,y) \coloneqq A\alpha_3^{k_3} x^{k_1}y^{k_2} -B(J-\alpha_1x-\alpha_2y)^{k_3}. \]
Suppose $f_3$ has a linear factor. We may consider factors of the form $y-Dx-E$, $x-D$ or $y-D$ for some rational numbers $D$ and $E$. Note that $x$ and $y$ are not factors of $f_3$, since $f_3(0, 0) = -J^{k_3} \neq 0$. Therefore, we may assume $D\ne 0$.

First, suppose that $y-Dx-E \mid f_3$. We obtain \[
A\alpha_3^{k_3} x^{k_1}(Dx+E)^{k_2} =B(J-\alpha_1x-\alpha_2(Dx+E))^{k_3}
\] for all integers $x$. Substituting $x=0$, we obtain $J=\alpha_2E$, which implies \[
A\alpha_3^{k_3} x^{k_1}(Dx+E)^{k_2} =B(-\alpha_1-\alpha_2D)^{k_3}x^{k_3}.
\] The right hand side is a monomial, while the left hand side is a monomial if and only if $E = 0$, since $A, \alpha_3, D, k_2 \neq 0$. But $E = 0$ implies $J = 0$; contradiction.

We now consider the other case, where either $x-D \mid f_3$ or $y-D \mid f_3$ for some rational number $D \neq 0$. Without loss of generality, assume that $y-D \mid f_3$. We now consider the integer points that lies in $y-D$; in particular, we assume that $D$ is an integer. In this case, we obtain
\begin{equation}\label{eqn:curve-k=3} A\alpha_3^{k_3} x^{k_1}D^{k_2} = B(J-\alpha_1x-\alpha_2D)^{k_3}. \end{equation}
for all integers $x$. Comparing coefficients, we obtain $J=\alpha_2D$. In particular, the only possible linear factor of $f_3$ of the form $y - D$ is $y - J/\alpha_2$. By symmetry, the only possible linear factor of $f_3$ of the form $x - D$ is $x - J/\alpha_1$.

From the previous paragraph, the only possible linear factors of $f_3$ are $x - J/\alpha_1$ and $y - J/\alpha_2$. Thus we can write
\[ f_3(x, y) = (x - J/\alpha_1)^{t_1} (y - J/\alpha_2)^{t_2} f_4(x, y) \]
for some integers $t_1, t_2 \geq 0$ and $f_4 \in \Z[x, y]$ such that $f_4$ has no linear factors. 

We now return to our proposition, which is equivalent to counting of integer points on $f$ with $x \ne J/\alpha_1$, $y\ne J/\alpha_2$. All such points must lie in $f_4$. There, we apply Proposition~\ref{prp:CCDN} with $\deg f_4\le \deg f_3 \le k_1+k_2+k_3$, which completes the proof of Theorem~\ref{thm:bompil-3var}.

We end this discussion by  noting that when $A=B=1$, the condition~\eqref{eqn:curve-k=3} may hold only if $\alpha_2\mid J$ and $|\alpha_3|\mid |\alpha_1|$.  These facts will be used in the Section~\ref{sec:2<=k<=4} for the case $k=3$.
In particular, we conclude that the condition $\alpha_1\nu_1,\alpha_2\nu_2\ne J$ always holds when $|\alpha_2|\nmid J$ or $|\alpha_3|$ is not a divisor of $|\alpha_1|$ or $|\alpha_2|$. In addition, if $|\alpha_3|\in \{|\alpha_1|,|\alpha_2|\}$, then all possible linear factors of $f_3$ are $x\pm 1$ and $y\pm 1$.

\subsubsection{Proof of Theorem~\ref{thm:bompil-4var}}

Without loss of generality, assume that $\alpha_1, \alpha_2, \alpha_3 \neq 0$. We first fix $\nu_4=W$. If $J-\alpha_4 W=0$ (which is impossible if $\alpha_4=0$), we may see that any choice of $\nu_3$ generates at most $k_1 + k_2$ choices of $(\nu_1,\nu_2)$ satisfying $\nu_1^{k_1} \nu_2^{k_2} = \nu_3^{k_3} W^{k_4}$ and $\alpha_1 \nu_1 + \alpha_2 \nu_2 = -\alpha_3 \nu_3$. Therefore, we obtain $O((k_1 + k_2) H)$ solutions to~\eqref{eqn:fatal-4var} in this case.

We may now assume $J-\alpha_4W\ne 0$.  Then, by following the proof of Theorem~\ref{thm:bompil-3var}, we consider the curve
\[ f_5(x,y) \coloneqq \alpha_3^{k_3} x^{k_1}y^{k_2} -W^{k_4}(J-\alpha_4W-\alpha_1x-\alpha_2y)^{k_3}. \] and determine whether this curve has a linear factor with integer points. With similar arguments in the proof of Theorem~\ref{thm:bompil-3var}, since $W(J-\alpha_4 W)\ne 0$, then $f_5$ does not have any factors of the form $y-Dx-E$ with $D \neq 0$.

Next, we consider the case where $x-D \mid f_5$ or $y-D \mid f_5$ for a rational number $D$. Without loss of generality we may assume $y-D \mid f_5$. Then we get $J-\alpha_4W=\alpha_2D$, $k_1=k_3$, and
\[ \alpha_3^{k_3} D^{k_2}=W^{k_4}(-\alpha_1)^{k_3}. \]
Substituting $D=(J-\alpha_4 W)/\alpha_2$, we obtain
\[ \alpha_3^{k_3}(J-\alpha_4W)^{k_2} = \alpha_2^{k_2}W^{k_4}(-\alpha_1)^{k_3}. \]
Since $J, \alpha_1, \alpha_2, \alpha_3 \neq 0$ and $k_2, k_4 > 0$, by looking at the above equality as a polynomial equality with respect to $W$, we get that there are at most $k_2 + k_4$ choices of $W$ such that this scenario is possible. Each choice of $W$ gives a unique value of $\nu_2 = D = (J-\alpha_4W)/\alpha_2$. To summarise this paragraph, there exists at most
\[ (k_1 + k_4) + (k_2 + k_4) \leq 2(k_1 + k_2 + k_3 + k_4) \]
choices of $W$ such that $f_5$ contains a line, and in these cases, $f_5$ contains at most one line of the form $x - D$, one line of the form $y - D$, and no other lines.

Now we factor $f_5$ as follows:
\[ f_5(x, y) = (x - D_x)^{t_1} (y - D_y)^{t_2} f_6(x, y), \]
where $D_x, D_y$ are rational numbers, $t_1, t_2$ are nonnegative integers, and $f_6$ is a polynomial with no linear factors. Since $\deg(f_6) \leq \deg(f_5) \leq k_1 + k_2 + k_3 + k_4$, by applying Proposition~\ref{prp:CCDN} to $f_6$, we get that $f_5$ has at most
\begin{align*}
    & 2(2H + 1) + C_0(k_1+k_2+k_3+k_4) H^{1/2} (\log H+2) \\
    &\quad+ C_1(k_1+k_2+k_3+k_4)^3 H^{1/3} (\log H+k_1+k_2+k_3+k_4)
\end{align*}
integer points of height at most $H$. Furthermore, the $2(2H + 1)$ term disappears if $f_5$ has no linear factors, which is the case for all but $2(k_1 + k_2 + k_3 + k_4)$ choices of $W$. Summing across all possible choices of $W$ concludes the proof of Theorem~\ref{thm:bompil-4var}.

\section{From counting points to computing volumes}\label{sec:prelim-lat}
\subsection{Counting points on a shifted lattice}
An important point in our argument is counting integer solutions to
\[ \balpha \cdot \bnu = J, \]
for a fixed vector $\balpha$ and constant $J$, where the coordinates of $\bnu$ are bounded by parameters $\bH$, which is done in the following theorem.

\begin{theorem}\label{thm:shifted-lattice-box-count}
Let $\balpha \in \Z^n$ be a nonzero vector.
Let $\cB$ be a box in $\R^n$ whose side lengths are at most $H$.
Then for any integer $J$,
\[ \#\{\bnu \in \cB \cap \Z^n : \balpha \cdot \bnu = J\} =
\begin{cases}
    0 & \quad \text{if} \quad \gcd(\balpha) \nmid J, \\
     V_{\balpha}(\cB; J) + O(H^{n - 2}) & \quad \text{if} \quad \gcd(\balpha) \mid J,
\end{cases} \]
    where
\begin{equation}\label{eq:Valpha-BJ}
    V_{\balpha}(\cB; J) = \dfrac{\gcd(\balpha) \Vol_{n - 1}(\{\bnu \in \cB : \balpha \cdot \bnu = J\})}{\|\balpha\|}.
\end{equation}
\end{theorem}

The rest of this section is dedicated for the proof of Theorem~\ref{thm:shifted-lattice-box-count}. First, note that the set $\{\bnu \in \Z^n : \balpha \cdot \bnu = J\}$ is nonempty if and only if $\gcd(\balpha)\mid J$, to which we may now assume. We then need to derive a result on approximating the number of lattice points in a convex body, for which we use the following result of Davenport.

\begin{lemma}\label{lem:Davenport-std-lattice}
For any convex body $\cB \subseteq \R^n$, we have
\[ \left|\#(\cB \cap \Z^n) - \Vol_n(\cB)\right| \leq \sum_{\varnothing \neq S \subseteq \{1, 2, \ldots, n\}} \Vol_{|S|}(\overline{\cB}_S), \]
    where for any $k \geq 0$, $\Vol_k$ denotes the $k$-dimensional volume, and for any $S \subseteq \{1, 2, \ldots, n\}$, $\overline{\cB}_S$ denotes the projection of $\cB$ into the subspace $\{\bnu \in \R^n : \forall i \in S, \nu_i = 0\}$.
\end{lemma}
\begin{proof}
Follows from~\cite{Dav}.
\end{proof}

Consider the lattice
\[ \Lambda = \{\bnu \in \Z^n : \balpha \cdot \bnu = 0\}. \]
Note that $\dim \Lambda=n-1$, thus there exists a bijective linear operator $T : \R^n \to \R^n$ that induces a bijection from $\Z^{n - 1}$ to $\Lambda$. Now, choose some $\bnu_0\in \Z^n$ with $\balpha \cdot \bnu_0 = J$. Then we have
\[ \{\bnu \in \cB \cap \Z^n : \balpha \cdot \bnu = J\} = \cB \cap (\bnu_0 + \Lambda). \]
Next, note the map $\phi : \R^n \to \R^n$ defined by $\phi(\bnu) \coloneqq T^{-1}(\bnu - \bnu_0)$ is a bijection between the sets $\cB \cap (\bnu_0 + \Lambda)$ and $T^{-1}(-\bnu_0 + \cB) \cap \Z^{n - 1}$, and $T^{-1}(-\bnu_0 + \cB)$ is a convex body. Furthermore, $T^{-1}(-\bnu_0 + \cB)$ is contained in a box whose side lengths are $O(H)$, since $\cB$ is a box whose side lengths are at most $H$.

Thus, by Lemma~\ref{lem:Davenport-std-lattice}, we obtain
\begin{align*}
    \#\{\bnu \in \cB \cap \Z^n : \balpha \cdot \bnu = J\}
    &= \#(T^{-1}(-\bnu_0 + \cB) \cap \Z^{n - 1}) \\
    &= \Vol_{n - 1}(T^{-1}(-\bnu_0 + \cB) \cap \R^{n - 1}) + O(H^{n - 2}) \\
    &= \lvert\det(T)\rvert^{-1} \Vol_{n - 1}(\cB \cap (\bnu_0 + T(\R^{n - 1}))) + O(H^{n - 2}).
\end{align*}
Since $T(\Z^{n - 1}) = \Lambda$ and $\balpha \cdot \bnu_0 = J$, we have
\[ \cB \cap (\bnu_0 + T(\R^{n - 1})) = \{\bnu \in \cB : \balpha \cdot \bnu = J\}. \]
Since $T$ is a bijection between $\Z^{n - 1}$ and $\Lambda$, we have
\[ \lvert\det(T)\rvert = d(\Lambda), \]
where $d(\Lambda)$ is the covolume of $\Lambda$. Thus, it now remains to prove
\begin{equation}\label{eqn:dLambda} d(\Lambda) = \gcd(\balpha)^{-1} \|\balpha\|. \end{equation}
By scaling, we may assume $\gcd(\balpha) = 1$.
Note  that since $\balpha$ is orthogonal to $\Lambda$, a fundamental parallelpiped for the lattice $\Lambda \oplus \balpha \Z$ is given by a right prism with height $\|\balpha\|$ and base being a fundamental parallelpiped for $\Lambda$.
Thus we have
\[ d(\Lambda \oplus \balpha \Z) = \|\balpha\| \, d(\Lambda). \]

On the other hand, since $\gcd(\balpha) = 1$, the map from $\Z^n$ to $\Z$ defined by $\bnu \mapsto \balpha \cdot \bnu$ is surjective, and the kernel is $\Lambda$ by definition.
Thus we have an isomorphism $\phi : \Z^n/\Lambda \cong \Z$ such that $\phi(\bnu \bmod{\Lambda}) = \balpha \cdot \bnu$ for all $\bnu \in \Z^n$.
Hence,
\[ \Z^n/(\Lambda \oplus \balpha \Z) \cong \Z/\phi(\balpha \bmod{\Lambda}) \, \Z = \Z/\|\balpha\|^2 \Z. \]
This implies $d(\Lambda \oplus \balpha \Z) = \|\balpha\|^2$.
Since we also have $d(\Lambda \oplus \balpha \Z) = \|\balpha\| \, d(\Lambda)$, we obtain $d(\Lambda) = \|\balpha\|$, which completes the proof of~\eqref{eqn:dLambda} and Theorem~\ref{thm:shifted-lattice-box-count}.

\subsection{Computing volumes on some regions in a box}\label{sec:count-vol}

    Theorem~\ref{thm:shifted-lattice-box-count} reduces our problem to computing the volumes of sections of hypercubes. We discuss several related results on this problem, beginning with a result from Marichal and Mossinghoff~\cite{MarMos} concerning the volume of sections of the standard hypercube, $[0,1]^n$.
    
\begin{theorem}\label{thm:Marichal-Mossinghoff}(\cite[Theorem 4]{MarMos})
Let $\balpha \in \R^n$ be a vector with nonzero coordinates.
Then for any real number $r$,
\[ \Vol_{n - 1}\left(\{\bnu \in [0, 1]^n : \balpha \cdot \bnu = r\}\right) = \frac{\|\balpha\|}{(n - 1)! \prod_i \alpha_i} \sum_{\mathbf{c} \in \{0, 1\}^n} (-1)^{\# \{i : c_i = 1\}} \max\{r - \balpha \cdot \mathbf{c}, 0\}^{n - 1}. \]
\end{theorem}
Using this result, we derive a new result over the hypercube $[-1/2, 1/2]^n$.
\begin{theorem}\label{thm:Marichal-Mossinghoff-half}
Let $\balpha \in \R^n$ be a vector with nonzero coordinates.
Then for any real number $r$,
\begin{multline*}
    \Vol_{n - 1}\left(\{\bnu \in [-1/2, 1/2]^n : \balpha \cdot \bnu = r\}\right) \\
    = \frac{\|\balpha\|}{2^{n - 1} (n - 1)! \prod_i \alpha_i} \sum_{\mathbf{c} \in \{-1, 1\}^n} (-1)^{\# \{i : c_i = 1\}} \max\{2r - \balpha \cdot \mathbf{c}, 0\}^{n - 1}.
\end{multline*} 
\end{theorem}
\begin{proof}
Indeed, Theorem~\ref{thm:Marichal-Mossinghoff} implies
\begin{align*}
    & \Vol_{n - 1}\left(\{\bnu \in [-1/2, 1/2]^n : \balpha \cdot \bnu = r\}\right) \\
    =\;& \Vol_{n - 1}\left(\{\bnu \in [0, 1]^n : \balpha \cdot \bnu = r + \balpha \cdot \bfone_n/2\}\right) \\
    =\;& \frac{\|\balpha\|}{(n - 1)! \prod_i \alpha_i} \sum_{\mathbf{c} \in \{0, 1\}^n} (-1)^{\# \{i : c_i = 1\}} \max\{r - \balpha \cdot (\mathbf{c} - \bfone_n/2), 0\}^{n - 1} \\
    =\;& \frac{\|\balpha\|}{(n - 1)! \prod_i \alpha_i} \sum_{\mathbf{c} \in \{-1/2, 1/2\}^n} (-1)^{\# \{i : c_i = 1/2\}} \max\{r - \balpha \cdot \mathbf{c}, 0\}^{n - 1} \\
    =\;& \frac{\|\balpha\|}{2^{n - 1} (n - 1)! \prod_i \alpha_i} \sum_{\mathbf{c} \in \{-1, 1\}^n} (-1)^{\# \{i : c_i = 1\}} \max\{2r - \balpha \cdot \mathbf{c}, 0\}^{n - 1},
\end{align*}
    where $\bfone_n$ is the all-one vector of dimension $n$.
\end{proof}
From these results, we obtain the following corollary for approximating the volume of a section of the hypercubes $[0, 1]^n$ and $[-1/2, 1/2]^n$ with the volume of another section that passes through the origin.

\begin{corollary}\label{cor:Marichal-Mossinghoff-approx}
Let $\balpha \in \R^n$ be a nonzero vector.
Then for any real number $r$, we have
\begin{align*}
    \Vol_{n - 1}\left(\{\bnu \in [0, 1]^n : \balpha \cdot \bnu = r\}\right) &= \Vol_{n - 1}\left(\{\bnu \in [0, 1]^n : \balpha \cdot \bnu = 0\}\right) + O(|r|), \\
    \Vol_{n - 1}\left(\{\bnu \in [-1/2, 1/2]^n : \balpha \cdot \bnu = r\}\right) &= \Vol_{n - 1}\left(\{\bnu \in [-1/2, 1/2]^n : \balpha \cdot \bnu = 0\}\right) + O(|r|).
\end{align*}
Furthermore, if $\balpha$ has at least three nonzero coordinates, then
\[ \Vol_{n - 1}\left(\{\bnu \in [-1/2, 1/2]^n : \balpha \cdot \bnu = r\}\right) = \Vol_{n - 1}\left(\{\bnu \in [-1/2, 1/2]^n : \balpha \cdot \bnu = 0\}\right) + O(|r|^2). \]
\end{corollary}
\begin{proof}
First suppose that $\balpha$ has no zero coordinates.
The expressions on the right hand side of the formulae in Theorem~\ref{thm:Marichal-Mossinghoff} and Theorem~\ref{thm:Marichal-Mossinghoff-half} are continuous, piecewise linear if $n = 2$, and differentiable if $n \geq 3$, while the left hand side is bounded.
Furthermore, as a function in terms of $r$, the expression
\[ \Vol_{n - 1}\left(\{\bnu \in [-1/2, 1/2]^n : \balpha \cdot \bnu = r\}\right) \]
is symmetric with respect to $0$.
This completes the proof when $\balpha$ has no zero coordinates.

In the case where $\balpha$ has some zero coordinates, we may instead consider the induced vector $\balpha'$ that consists of all nonzero coordinates of $\balpha$ and apply our arguments here.
\end{proof}
With regards to our arguments, we apply Corollary~\ref{cor:Marichal-Mossinghoff-approx} to the variable $V_{\balpha}(\cB; J)$ defined in~\eqref{eq:Valpha-BJ} to obtain the following result.

\begin{corollary}\label{cor:Valpha-BJ-approx}
Let $\balpha \in \R^n$ be a nonzero vector with $k > 0$ nonzero coordinates.
Then for any real numbers $H \geq 1$ and $J$,
\[ V_{\balpha}([0, H]^n; J) = V_{\balpha}([0, 1]^n; 0) H^{n - 1} + O(|J| H^{n - 2}) \]
    and
\[ V_{\balpha}([-H, H]^n; J) = \begin{cases}
    V_{\balpha}([-1/2, 1/2]^n; 0) (2H)^{n - 1} + O(|J| H^{n - 2}) & \quad \text{if} \quad k \leq 2, \\
    V_{\balpha}([-1/2, 1/2]^n; 0) (2H)^{n - 1} + O(|J|^2 H^{n - 3}) & \quad \text{if} \quad k \geq 3.
\end{cases} \]
\end{corollary}
\begin{proof}
The only thing to note is that~\eqref{eq:Valpha-BJ} implies $V_{\alpha}(H \cB; 0) = H^{n - 1} V_{\alpha}(\cB; 0)$ for any box $\cB$.
\end{proof}

In the special case where all coordinates of $\balpha$ are positive, we use the following lemma on computing the volume of the related region.

\begin{lemma}
For any $\balpha \in \R^n$ with positive coordinates and for any $r \in \R$,
\[ \Vol_{n - 1}(\{\bnu \in \R^n : \bnu \geq \bfzero, \balpha \cdot \bnu = r\}) = \frac{\|\balpha\|}{(n - 1)! \prod_i \alpha_i} \max\{r, 0\}^{n - 1}. \]
\end{lemma}
\begin{proof}
Note that by scaling, we may assume $r=1$. We now prove
\[ \frac{\|\balpha\|^{-1} \Vol_{n - 1}(\{\bnu \in \R^n : \bnu \geq \bfzero, \balpha \cdot \bnu = 1\})}{n} = \frac{1}{n! \prod_i \alpha_i}. \]
We view both sides as the volume of the simplex whose vertices are $\bfzero$ and $\alpha_1^{-1} \bfe_1, \ldots, \alpha_n^{-1} \bfe_n$, where the $\bfe_i$'s are the standard basis vectors of $\R^n$.
This is obvious for the right hand side, and for the left hand side, we just need to check that the distance from $\bfzero$ to the face $\{\bnu \in \R^n : \bnu \geq \bfzero, \balpha \cdot \bnu = 1\}$ is $\|\balpha\|^{-1}$.
This holds since the projection of $\bfzero$ to this face is $\|\balpha\|^{-2} \balpha$.
\end{proof}

Thus in the case where $\balpha$ has positive coordinates, we obtain another formula for $V_{\balpha}([0, H]^n; J)$ when $J\ll H$.

\begin{corollary}\label{cor:Valpha-BJ-pos}
For any $\balpha \in \R^n$ with positive coordinates and for any real numbers $H \geq 1$ and $J \geq 0$ with $J \le H$,
\[ V_{\balpha}([0, H]^n; J) = \frac{\gcd(\balpha)}{(n - 1)! \prod_i \alpha_i} J^{n - 1}. \]
\end{corollary}
The lack of $H$ variables in this corollary is expected, since all points in the corresponding bodies are bounded by, for example, the box $[0,J]^n$.

\section{Multiplicatively dependent vectors of large rank}\label{sec:upper-bounds}

Our main strategy of proving our main results is counting the number of vectors in $\cS_n(H,J;\balpha)$ based on its multiplicative rank, defined as follows.
\begin{definition}\label{def:multrank}
Let $\overline{\Q}$ be an algebraic closure of the rational numbers $\Q$. For each $\bm{\nu}$ in $(\overline{\Q^*})^n$, define $s$ as the \textit{(multiplicative) rank} of $s$ in the following way. If $\bm{\nu}$ has a coordinate which is a root of unity, we put $s=0$. Otherwise, $s$ is defined as be the largest integer with $1\le s \le n$ such that any $s$ coordinates of $\bm{\nu}$ form a multiplicatively independent vector.
\end{definition}

 Note that the rank $s$ of a multiplicatively dependent vector $\bm\nu \in (\overline{\Q^*})^n$, which is our main object, satisfies \begin{align*}
    0\le s\le n-1.
\end{align*}
 Now, denote
\[ S_{n,r}(H,J;\balpha)=\#\{ \bm{\nu}\in \cS_n(H,J;\balpha) \colon \text{ the multiplicative rank of }\bnu\text{ is } r\}. \]
From this definition, we obtain the following identity: \begin{equation}\label{eqn:sumrankD}
    S_{n}(H,J;\balpha)=S_{n,0}(H,J;\balpha)+S_{n, 1}(H,J;\balpha)+\dots+S_{n, n-1}(H,J;\balpha).
\end{equation}

We first obtain some upper bounds for $S_{n,r}(H,J;\balpha)$. These bounds hold uniformly with respect to $J$ and $k$, the number of nonzero coordinates in $\balpha$. We first state our bound in the following form, which corresponds to the indices of the corresponding multiplicatively dependent set.

\begin{lemma}\label{lem:SnIupper}
Let $n$ be a positive integer, $\balpha \in \Z^n$ be a nonzero vector, and $I$ be a subset of $\{1, 2, \ldots, n\}$.
Let $S_{n, I}(H, J; \balpha)$ denote the set of vectors $\bnu \in \cS_n(H, J; \balpha)$ such that the collection $(\nu_i)_{i \in I}$ is multiplicatively dependent but $(\nu_i)_{i \in I'}$ is multiplicatively independent for any proper subset $I' \subsetneq I$.
Then there exists constants $c, c_I > 0$, depending only on $n$ and $I$, such that
\[ \#S_{n,I}(H,J;\balpha) < \begin{cases}
        c_I H^{n-1-\lceil\#I/2\rceil}\exp(c\log H/\log\log H) & \quad \text{if} \quad \{i : a_i \neq 0\} \not\subseteq I, \\
        c_I H^{n-\lceil\#I/2\rceil}\exp(c\log H/\log\log H) & \quad \text{if} \quad \{i : a_i \neq 0\} \subseteq I.
    \end{cases} \]
\end{lemma}
\begin{proof}
Denote $t=\#I$. We first bound the number of choices of $\nu_i$ with $i\notin I$. Trivially, if we fix all the coordinates $\nu_i$ with $i \in I$, there are at most $(2H)^{n - t}$ choices for the coordinates $\nu_i$ with $i \notin I$. This completes our argument when $\{i : a_i \neq 0\} \subseteq I$.

Now, suppose that $\{i : a_i \neq 0\} \not\subseteq I$. Then the coordinates $\nu_i$ with $i \notin I$ satisfies
\[ \sum_{i \notin I} \alpha_i \nu_i = J - \sum_{i \in I} \alpha_i \nu_i. \] Since the left-hand side is nonzero, there are at most $(2H)^{n - t - 1}$ choices for the coordinates $\nu_i$ with $i \notin I$ for a fixed choice of $\nu_i$ with $i\in I$. This establishes our upper bound of $\nu_i$ with $i\notin I$.

Next, we bound the number of possible choices for $\nu_i$ with $i \in I$. Our argument mostly follows the proof of~\cite[Propositions~1.3 and~1.5]{PSSS}, however with some simplified arguments with respect to our setup.
Let $I=\{j_1,\dots,j_t\}$. Then, from the minimality of $I$, there exist nonzero integers $k_{j_1}$, $\dots$, $k_{j_t}$, with
\begin{equation}\label{eqn:sI}
    \nu_{j_1}^{k_{j_1}}\dots \nu_{j_t}^{k_{j_t}}=1.
\end{equation}Furthermore, using Lemma~\ref{lem:exp}, we may let
\begin{align*}
    \max \{|k_{j_1}|,\dots,|k_{j_t}| \} < c_3(\log H)^{t-1}.
\end{align*}

Let $P$ be the set of indices $i$ such that $k_i$ is positive, and $N$ be the set of indices $i$ for which $k_i$ is negative. We note that either $\#P$ or $\#N$ is at least $\lceil t/2\rceil $. Also, we may rewrite~\eqref{eqn:sI} as \begin{equation}\label{eqn:sIPN}    
\prod_{i\in P} \nu_{j_i}^{k_{j_i}}=\prod_{i\in N} \nu_{j_i}^{-k_{j_i}}.
\end{equation}

Based on this observation, we now let $I_0$ be the set of indices $i$ from $I$ such that $k_i$ is positive if $\# P\ge \lceil t/2\rceil $, and the the set of indices $i$ such that $k_i$ is negative otherwise. Let $t_0=\#I_0$. In either case we have
\[ t_0 \ge \lceil t/2\rceil. \]
We rewrite~\eqref{eqn:sIPN} as \begin{equation}\label{eqn:sI0}    
\prod_{i\in I_0} \nu_{i}^{|k_i|}=\prod_{i\in I\backslash I_0} \nu_{i}^{|k_i|}.
\end{equation}
Observe that the number of different equations that may arise from~\eqref{eqn:sI0} is at most
\[ (2c_3 (\log H)^{t-1})^t (2H)^{t-t_0} < c_4 H^{t-t_0} (\log H)^{c_5}, \]
for some positive constants $c_4$, $c_5$ that only depend on $n$.

We now fix $\nu_{i}$ for $i\in I \backslash I_0$ and bound the number of choices for $\nu_{i_0}$ for each $i_0 \in I_0$. By~\eqref{eqn:sI0}, $\nu_{i_0}$ divides $\prod_{i\in I\backslash I_0} \nu_i$, and $\left|\prod_{i\in I\backslash I_0} \nu_i\right| \leq H^n$. Therefore, since $|\nu_{i_0}| \leq H$, we use Lemma~\ref{lem:divdeBruijn-div} to conclude there are at most $\exp(c_6 \log H/\log \log H)$ choices of $\nu_{i_0}$ for some positive constant $c_6$ that only depend on $n$. Then there are at most 
\begin{equation}\label{eqn:vji}
    H^{t-t_0} \exp(c_7 \log H/\log \log H) \le H^{\lfloor t/2 \rfloor}  \exp(c_7 \log H/\log \log H) 
\end{equation}
ways of picking $\nu_{i_0}$ across all $i_0 \in I_0$, for some positive constant $c_7$ depending only on $n$. Multiplying~\eqref{eqn:vji} with our bound for $i\not\in I$ completes the proof.
\end{proof}

Applying Lemma~\ref{lem:SnIupper} with respect to our multiplicative rank notation, we obtain the following corollary.

\begin{corollary}\label{cor:Snrupper}
Let $n$ be a positive integer, and $\balpha \in \Z^n$ with $k$ nonzero coordinates. Then, for all nonnegative integers $r < n$,
\begin{equation*}
        S_{n,r}(H,J;\balpha)< \begin{cases}
            c_{r} H^{n-1-\lceil  (r+1)/2\rceil}\exp(c\log H/\log\log H) & \quad \text{if} \quad r\le k-2,\\
            c_{r} H^{n-\lceil  (r+1)/2\rceil}\exp(c\log H/\log\log H) &\quad \text{otherwise},
        \end{cases}
\end{equation*}
for constants $c,c_r>0$ that only depend on $n$ and $r$.
\end{corollary}
Note that for $r>k-2$, our bound is exactly the bound in~\cite[Equation (1.22)]{PSSS} for the number of multiplicatively dependent integer vectors without the hyperplane condition.

Applying this corollary separately for $k\ge 5$ and $k<5$, we obtain the following statement.

\begin{corollary}\label{cor:mainterm}
    Let $H$, $J$ be integers, $n>0$ be a positive integer and $\balpha\in \Z^n$ with $k$ nonzero coordinates. Then, as $H\to \infty$, 
\begin{equation}\label{eqn:sumrankDorder}\begin{split}
    &S_n(H,J;\balpha)= 
        S_{n,0}(H,J;\balpha)+S_{n,1}(H,J;\balpha)\\ &\quad +\begin{cases}
            S_{n,2}(H,J;\balpha)+S_{n,3}(H,J;\balpha) +O(H^{n-3+o(1)}) & \quad \text{if} \quad k< 5,\\
            O(H^{n-3+o(1)}) &\quad \text{if}\quad k\ge 5.
        \end{cases}
\end{split}
\end{equation}
\end{corollary}

Note that our heuristic expects the main term of $S_n(H,J;\balpha)$ to be of order $H^{n-2}$ for most cases. However, applying Corollary~\ref{cor:Snrupper} for $k<5$ only implies $S_{n,2}(H,J;\balpha),S_{n,3}(H,J;\balpha)\le H^{n-2+o(1)}$, which is not sufficient for our purposes. Therefore, we devote Section~\ref{sec:too-many-zeroes} to improve this bound.

On the other hand, we will compute the asymptotics for $S_{n,0}(H,J;\balpha)$ and $S_{n,1}(H,J;\balpha)$ in Section~\ref{sec:small-rank}, which will complete our computation for $S_{n}(H,J;\balpha)$.

\section{Hyperplanes with few nonzero coordinates}\label{sec:too-many-zeroes} \subsection{Introduction}

With regards of Corollary~\ref{cor:mainterm}, we now consider the case $k<5$. The case of $k=0$, which corresponds to $\balpha = \bfzero_n$, is trivial except when $J=0$. In this case, the problem becomes counting multiplicatively dependent integer vectors without any other condition, which is already done by Pappalardi, Sha, Shparlinski and Stewart in~\cite[Theorem~1.4]{PSSS}. Our focus in this section is computing/bounding $S_{n, r}(H, J; \balpha)$ for $r = 2, 3$ when $1 \leq k \leq 4$.

For the case $k=1$, which reduces to the case $\balpha = \bfe_1 = (1, 0, \ldots, 0)$, we obtain the following statement, which is proven in Section~\ref{sec:k=1}.

\begin{proposition}\label{prp:k=1}
Let $n \geq 3$, $J \neq 0$ and $H$ be integers. 
We have \begin{align*}
    S_n(H,J;\bfe_1)=\begin{cases} (2H)^{n-1} &\quad \text{when} \quad |J|=1,\\
        S_{n,0}(H,J;\bfe_1)+S_{n,1}(H,J;\bfe_1) \\ \: + \dfrac{2^{n-1}(n-1)(n-2)}{f(|J|)-1}H^{n-2}+O(H^{n-5/2}) &\quad \text{when} \quad |J|>1,
    \end{cases}
    \end{align*} where $f$ is defined in Theorem~\ref{thm:main-k=1}.
\end{proposition}

For the case $k=2,4$, we provide a power-saving improvement  of~\eqref{eqn:sumrankDorder} which is proven in Section~\ref{sec:2<=k<=4}.

\begin{proposition}\label{prp:k=2,3,4}
Let $\balpha\in \Z^n$ with $k$ nonzero coordinates, and let $J \ne 0$ and $H$ be integers. If $k=2,4$, then
\[ S_n(H,J;\balpha)=S_{n,0}(H,J;\balpha)+S_{n,1}(H,J;\balpha)+O(H^{n-5/2} (\log H)^{16}). \] If $k=3$, there exists a constant $C^{(2);k=3}_{\balpha;J}\ge 0$   with \[ S_n(H,J;\balpha)=S_{n,0}(H,J;\balpha)+S_{n,1}(H,J;\balpha)+C^{(2);k=3}_{\balpha;J}H^{n-2}+O(H^{n-5/2} (\log H)^{16}).
\] 
\end{proposition}
Note that for the case $k=3$, we also provide some possible conditions of $\balpha$ and $J$ such that $C^{(2);k=3}_{\balpha;J}=0$.

We now describe our main strategy for proving Proposition~\ref{prp:k=2,3,4}.
Due to Lemma~\ref{lem:exp}, to count $S_{n, 2}(H, J; \balpha)$, we need to count the number of vectors $\bnu \in \cS_n(H, J; \balpha)$ such that there exists distinct indices $j_1, j_2, j_3$ and positive integers $k_1, k_2, k_3 \ll (\log H)^2$ with
\begin{equation}\label{eqn:rank2}
\nu_{j_1}^{k_1}\nu_{j_2}^{k_2} =\nu_{j_3}^{k_3}.
\end{equation}
Meanwhile, to count $S_{n, 3}(H, J; \balpha)$, we need to count the number of vectors $\bnu \in \cS_n(H, J; \balpha)$ such that there exists distinct indices $j_1, j_2, j_3, j_4$ and positive integers $k_1, k_2, k_3, k_4 \ll (\log H)^3$ with either
\begin{equation}\label{eqn:rank3-case1}
    \nu_{j_1}^{k_1}\nu_{j_2}^{k_2} =\nu_{j_3}^{k_3}\nu_{j_4}^{k_4}
\end{equation}
    or
\begin{equation}\label{eqn:rank3-case2}
    \nu_{j_1}^{k_1}\nu_{j_2}^{k_2} \nu_{j_3}^{k_3}=\nu_{j_4}^{k_4}.
\end{equation}
Note that by the proof of Lemma~\ref{lem:SnIupper} (with $t_0 = 3$), there are at most
\begin{equation}\label{eqn:rank3-case2-count}
    \binom{n}{4} (\log H)^{12} \cdot O(H^{n - 3} \exp(c \log H/\log \log H)) = O(H^{n - 3 + o(1)})
\end{equation}
    vectors $\bnu \in S_n(H, J; \balpha)$ of rank $3$ satisfying~\eqref{eqn:rank3-case2} for some indices $j_1, j_2, j_3, j_4$ with $k_1, k_2, k_3, k_4 \ll (\log H)^3$.
Thus, to count vectors in $S_{n, 3}(H, J; \balpha)$, we only need to consider~\eqref{eqn:rank3-case1}.

\subsection{The case \texorpdfstring{$k=1$}{k=1}}\label{sec:k=1}

To count $S_n(H,J;\bfe_1)$, we need to  count multiplicatively dependent vectors $\bnu \in \Z^n$ with $0<|\nu_i|\le H$ for $2\le i \le n$ and $\nu_1=J$. When $|J|=1$, any arbitrary choices of other coordinates of $\bnu$ always guarantee a multiplicatively dependent vector. Thus, there are $(2H)^{n-1}$ choices of $\bnu$ in this case. Thus, we may now assume $|J|\ne 1$. Furthermore, we note that if $(J,\dots)$ is a multiplicatively dependent vector, then so does $(-J, \dots)$ and $(J^r,\dots)$ for any positive integer $r$. Thus, we may assume that $J > 1$ and that $J$ is not a perfect power.

We first handle the case $r=3$. Consider a vector $\bnu \in \cS_n(H,J;\bfe_1)$ of rank 3; thus, this vector satisfies~\eqref{eqn:rank3-case1} or~\eqref{eqn:rank3-case2}. As mentioned before, we only need to consider the case where the vector satisfies~\eqref{eqn:rank3-case1}. If $1\notin \{j_1,j_2,j_3,j_4\}$, Lemma~\ref{lem:SnIupper} implies that there are at most $O(H^{n-3+o(1)})$ possible choices for the vector $\bnu$.

Next, suppose $1 \in \{j_1,j_2,j_3,j_4\}$. Without loss of generality, let $j_4=1$, which implies
\[ \nu_1^{k_1}\nu_2^{k_2}=\nu_3^{k_3}J^{k_4}. \]
We fix $\nu_3$ in $H$ ways and use Lemma~\ref{lem:divdeBruijn-div} to obtain at most $H^{o(1)}$ choices for $(\nu_1,\nu_2)$. We then apply Lemma~\ref{lem:exp} for bounding the number of exponent choices, and choose the other coordinates of $\bnu$ in $H^{n-4}$ ways. Hence, in the case $j_4=1$, we also obtain at most $O(H^{n-3+o(1)})$ possible vectors $\bnu$. Combining with~\eqref{eqn:rank3-case2-count}, we get
\[ S_{n,3}(H,J;\bfe_1) = O(H^{n-3+o(1)}). \]

Now, we turn our attention to the case $r=2$. Note that in this section, the term $S_{n,2}(H,J;\bfe_1)$ also contributes in the main term of $S_n(H,J;\bfe_1)$, as can be seen with tuples of the form $(J,x,Jx,\dotsc) \in \cS_n(H,J;\bfe_1)$ for any integer $0<|x|\le H/J$.  Thus, we need to compute the correct asymptotic for this quantity.

Following~\eqref{eqn:rank2}, we need to count $\bnu \in \cS_n(H,J;\bfe_1)$ such that there exist some indices $j_1$, $j_2$, $j_3$ and positive integers $k_1, k_2, k_3 \ll (\log H)^2$ with
\begin{equation}\label{eqn:k=1, r=2} \nu_{j_1}^{k_1}\nu_{j_2}^{k_2} =\nu_{j_3}^{k_3}. \end{equation}
Furthermore, we need to restrict $\bnu$ so that $|\nu_i| \neq 1$ for any $i \leq n$ and $|\nu_{i_1}|^{k_1} \neq |\nu_{i_2}|^{k_2}$ for any positive integers $k_1, k_2$ whenever $i_1 \neq i_2$, so that $\bnu$ has rank $2$. Again, we first fix these exponents.

By Lemma~\ref{lem:SnIupper}, when $1\notin \{j_1,j_2,j_3\}$, there are at most $O(H^{n-4} H^{1+o(1)})=O(H^{n-3+o(1)})$ possible vectors $\bnu$. Thus, we may assume $1 \in \{j_1,j_2,j_3\}$. However, when $j_3=1$, we may repeat the argument to obtain at most $O(H^{n-3+o(1)})$ possible vectors $\bnu$. Hence, we may assume $j_1=1$ and obtain 
\begin{equation}\label{eqn:j1=1}
    J^{k_1}\nu_{j_2}^{k_2} = \nu_{j_3}^{k_3}.
\end{equation}
We may fix $j_2$ and $j_3$ and assume that $\nu_{j_2},\nu_{j_3}>0$ and $\gcd(k_2,k_3)=1$. Thus, we need to count the number of possible solutions $(\nu_{j_2},\nu_{j_3})$ to~\eqref{eqn:j1=1} for a fixed choice of $(k_1,k_2,k_3)$, such that the corresponding vector $\bnu$ is of rank exactly $2$. Note that in particular, the rank condition implies $\nu_{j_2}, \nu_{j_3} > 1$.

We first consider the case where $k_2=k_3$, which easily implies $(\nu_{j_2},\nu_{j_3})=(x,J^rx)$ for an integer $r > 0, x>1$ with $J^r x \le H$. Thus, up to sign changes, our problem now is equivalent to counting pairs $(J^r,x)$ such that $x$ is not a power of $J$ and $J^r x \le H$. Then, there are
\[ \sum_{0 < r < \log H/\log J} \left(\frac{H}{J^r} + O\left(\frac{\log H}{\log J}\right)\right) = \dfrac{H}{J-1}+ O((\log H)^2) \]
possible choices of $(J^r,x)$ in this case for a fixed $J$ and positive $x$.
Considering all possible signs of these coordinates, we obtain that there are
\begin{equation*}
    \dfrac{4}{J-1} H + O((\log H)^2)
\end{equation*}
possible choices for $(\nu_{j_2},\nu_{j_3})$ in this case.

We now consider the case when $k_2\ne k_3$. Suppose that $k_2>k_3$. Then, $\nu_{j_3}^{k_3}/J^{k_1}$ is a $k_2$-th perfect power. Thus, there are
\begin{equation*}
\sum_{2 \le k_2 \ll (\log H)^2} O(H^{1/k_2}) = O(H^{1/2}) \end{equation*}
possible choices for $(\nu_{j_2},\nu_{j_3})$ in this case. We also obtain a similar upper bound when $k_2<k_3$. These arguments imply there are \begin{equation}\label{eqn:k=1, r<=2}
    \dfrac{4}{J-1} H(2H - 2)^{n-3} + O(H^{n-5/2})= \dfrac{2^{n-1}}{J-1}H^{n-2}+O(H^{n-5/2})
\end{equation} vectors that satisfy~\eqref{eqn:j1=1} for a fixed $j_2$ and $j_3$ with rank at least $1$.

Now we bound and exclude the number of vectors satisfying~\eqref{eqn:k=1, r=2} with rank exactly $1$ that satisfy~\eqref{eqn:k=1, r=2}, which means the number of vectors $\bnu$ that satisfy~\eqref{eqn:k=1, r=2} and \begin{equation*}
\nu_{j_4}^{k_4}=\nu_{j_5}^{k_5}
\end{equation*} for some indices $j_4,j_5$ and positive integers $k_4,k_5\ll \log H$. Note that this system of equation consists of at least three different variables.  Following the arguments leading to~\eqref{eqn:j1=1}, we may let $j_1=1$ and replace~\eqref{eqn:k=1, r=2} with \eqref{eqn:j1=1}. 

We note that if $\{j_4,j_5\}\not\subseteq \{1,j_2,j_3\}$, there are trivially $O(H^{n-3+o(1)})$ vectors $\bnu$ that satisfy this case by fixing the variables in each of the equations. Therefore,  we may now assume $\{j_4,j_5\}\subseteq \{1,j_2,j_3\}$. This implies (without loss of generality) either \begin{equation*}
    J^{k_5} = \nu_{j_2}^{k_4}  \quad \text{or} \quad \nu_{j_2}^{k_4} = \nu_{j_3}^{k_5}, 
\end{equation*} In either of these cases, $\nu_{j_2}$ and $\nu_{j_3}$ are uniquely determined from this equation and~\eqref{eqn:j1=1}, for a fixed $k_1$, $\dots$, $k_5$. Therefore, adding over all possible exponents, there are at most $O(H^{n-3}(\log H)^8)$ vectors that have to be excluded.

Returning to the original counting, we note there are $(n-1)(n-2)$ choices for $(j_2,j_3)$. Therefore, returning to~\eqref{eqn:k=1, r<=2}, we obtain
\[ S_{n,2}(H,J;\bfe_1)=\dfrac{2^{n-1}(n-1)(n-2)}{J-1}H^{n-2}+O(H^{n-5/2}). \]
    when $n\ge 4$. This completes the proof of Proposition~\ref{prp:k=1}.

\subsection{The case \texorpdfstring{$2\leq k\leq 4$}{2<=k<=4}}\label{sec:2<=k<=4}

Instead of using Corollary \ref{cor:Snrupper} to bound $S_{n,r}$ in the case $k=2,3,4$, we use  bounds that are derived from Theorems \ref{thm:bompil-2var}, \ref{thm:bompil-3var}, and \ref{thm:bompil-4var}, respectively. These arguments yield Proposition \ref{prp:k=2,3,4}.

\subsubsection{The case $k=2$} We note that for all $\bnu\in \cS_n(H,J;\balpha)$, we have
\begin{equation}\label{eqn:alpha-k=2}
\alpha_{i_1}\nu_{i_1}+\alpha_{i_2}\nu_{i_2}=J,
\end{equation}
for some indices $i_1$ and $i_2$ such that the corresponding coefficients in $\balpha$ are nonzero.

We first consider the case $r=2$, and consider the related vector $\bnu$ of rank 2. Then $\bnu$ satisfies~\eqref{eqn:rank2} for some $j_1, j_2, j_3$ and positive integers $k_1, k_2, k_3 \ll (\log H)^2$. Applying Lemma~\ref{lem:SnIupper}, when $\{i_1, i_2\} \not\subseteq \{j_1, j_2, j_3\}$, the number of such vectors across all exponents $k_1, k_2, k_3$ is at most
\[ n^3 (\log H)^6 \cdot O(H^{n - 3 + o(1)}) = O(H^{n - 3 + o(1)}). \]
    
We now assume that $\{i_1, i_2\} \subseteq \{j_1, j_2, j_3\}$. Without loss of generality, we may work with the system of equations \begin{align*}
    &\nu_{j_1}^{k_1}\nu_{j_2}^{k_2}=\nu_{j_3}^{k_3} \quad \text{or} \quad \nu_{j_1}^{k_1}\nu_{j_3}^{k_2}=\nu_{j_2}^{k_3} ,\\
    &\alpha_{j_1}\nu_{j_1}+\alpha_{j_2}\nu_{j_2}=J.
\end{align*}
In this case, recalling that $k_1,k_2,k_3\ll (\log H)^2$, we apply Theorem~\ref{thm:bompil-2var} to obtain at most $$O((k_1+k_2+k_3)H^{1/2} \log H)$$ tuples $(\nu_{j_1},\nu_{j_2},\nu_{j_3})$ for a fixed $k_1,k_2,k_3$. Adding over all $k_1,k_2,k_3$, we obtain at most $$O(H^{1/2}(\log H)^{9})$$ tuples $(\nu_{j_1},\nu_{j_2},\nu_{j_3})$ overall. Considering the other $n-3$ coordinates of $\bnu$, we obtain
\begin{equation}\label{eqn:k=2, r=2}
S_{n,2}(H,J;\balpha)\ll H^{n-5/2}(\log H)^{9}.
\end{equation} 

We proceed to the case $r=3$, where we instead consider the system of equations that comes from~\eqref{eqn:rank3-case1} and \eqref{eqn:alpha-k=2}. Again, by Lemma~\ref{lem:SnIupper}, when $\{i_1,i_2\}\not\subseteq \{j_1,j_2,j_3,j_4\}$, there are $O(H^{n-3+o(1)})$ possible choices for $\bnu$.

Therefore, we let $\{i_1,i_2\}\subseteq \{j_1,j_2,j_3,j_4\}$. In this case, we may first fix (without loss of generality) $\nu_{j_4}$, with $j_4\ne i_1,i_2$, in $H$ ways. Then, for the rest of the coordinates, we recall $k_1,\dots,k_4\ll (\log H)^3$ and use Theorem~\ref{thm:bompil-2var} to obtain at most $$O((k_1+k_2+k_3+k_4) H^{1/2} \log H)$$ solutions for a fixed $k_1,k_2,k_3,k_4$.  Adding over all possible $(k_1,k_2,k_3,k_4)$, we obtain there are at most \[
O(H^{3/2}(\log H)^{16})
\] choices for the tuple $(\nu_{j_1}, \nu_{j_2}, \nu_{j_3}, \nu_{j_4})$. By considering the other $n-4$ coordinates of $\bnu$, we obtain \begin{equation}\label{eqn:k=2, r=3}
S_{n,3}(H,J;\balpha)\ll H^{n-5/2}(\log H)^{16}.
\end{equation}
Substituting~\eqref{eqn:k=2, r=2} and~\eqref{eqn:k=2, r=3} to~\eqref{eqn:sumrankDorder} completes the proof of Proposition~\ref{prp:k=2,3,4} for $k=2$.

\subsubsection{The case $k=4$}  We recall that Corollary~\ref{cor:Snrupper} and~\eqref{eqn:sumrankD} implies \[
S_n(H,J;\balpha)=S_{n,0}(H,J;\balpha)+S_{n,1}(H,J;\balpha)+S_{n,3}(H,J;\balpha)+O(H^{n-3+o(1)}).
\] We now derive an upper bound for $S_{n,3}(H,J;\balpha)$. Now, note that for all $\bnu\in \cS_n(H,J;\balpha)$,  \begin{equation}\label{eqn:alpha-k=4}
\alpha_{i_1}\nu_{i_1}+\alpha_{i_2}\nu_{i_2}+\alpha_{i_3}\nu_{i_3}+\alpha_{i_4}\nu_{i_4}=J,
\end{equation}
for some indices $i_1$, $i_2$, $i_3$, $i_4$ such that the corresponding coefficients in $\balpha$ are nonzero. By similar argument as in the case $k = 2$ for $r = 3$, the problem reduces to bounding the number of solutions to~\eqref{eqn:rank3-case1} and~\eqref{eqn:alpha-k=4} with $\{i_1, i_2, i_3, i_4\} = \{j_1, j_2, j_3, j_4\}$.

We apply Theorem~\ref{thm:bompil-4var} to this equation for each choice of $(k_1,k_2,k_3,k_4)$. Arguing similarly as in the previous cases, we obtain that there are at most $O(H^{n-5/2}(\log H)^{16})$ choices of $\bnu$ for this case. Hence, \[
S_{n,3}(H,J;\balpha)\ll H^{n-5/2} (\log H)^{16}
\]in this case, which in turns also imply \begin{equation*}
S_{n}(H,J;\balpha)=S_{n,0}(H,J;\balpha)+S_{n,1}(H,J;\balpha)+O(H^{n-5/2} (\log H)^{16}).
\end{equation*}This completes the proof of Proposition~\ref{prp:k=2,3,4} for $k=4$.

\subsubsection{The case $k=3$}
In this case, we note that for all $\bnu\in \cS_n(H,J;\balpha)$,  \begin{equation}\label{eqn:alpha-k=3}
\alpha_{i_1}\nu_{i_1}+\alpha_{i_2}\nu_{i_2}+\alpha_{i_3}\nu_{i_3}=J,
\end{equation} for some indices $i_1$, $i_2$, $i_3$ such that the corresponding coefficients in $\balpha$ are nonzero.

We first consider the case $r=2$, and consider the related vector $\bnu$ of rank 2. With similar argument as in the case $k=2$, the problem reduces to counting the number of integer solutions to~\eqref{eqn:rank2} and~\eqref{eqn:alpha-k=3} with $\{i_1, i_2, i_3\} = \{j_1, j_2, j_3\}$ and $1<|\nu_{i_1}|,|\nu_{i_2}|,|\nu_{i_3}|\le H$. First, we bound integer solutions with $\nu_{i_1} \alpha_{i_1}, \nu_{i_2} \alpha_{i_2}, \nu_{i_3} \alpha_{i_3} \neq J$.

Without loss of generality, we consider the system of equations
\begin{align*}
    &\nu_{i_1}^{k_1}\nu_{i_2}^{k_2}=\nu_{i_3}^{k_3},\\
    &\alpha_{i_1}\nu_{i_1}+\alpha_{i_2}\nu_{i_2}+\alpha_{i_3}\nu_{i_3}=J,
\end{align*}
for a fixed $k_1$, $k_2$, $k_3$, with $\alpha_{j_1} \nu_{j_1}, \alpha_{j_2} \nu_{j_2}, \alpha_{j_3} \nu_{j_3} \neq J$. Arguing similarly as the case $k=2$ but using Theorem~\ref{thm:bompil-3var}, we obtain at most \[
O(H^{n-5/2}(\log H)^{9})
\] such vector $\bnu$ with $\alpha_{i_1} \nu_{i_1}, \alpha_{i_2} \nu_{i_2}, \alpha_{i_3} \nu_{i_3} \neq J$. In particular, following the proof of Theorem~\ref{thm:bompil-3var}, if one of these is true: \begin{itemize}
    \item None of $|\alpha_{i_1}|$, $|\alpha_{i_2}|$, $|\alpha_{i_3}|$ are divisors of the other numbers,
    \item $\alpha_1,\alpha_2,\alpha_3 \nmid J$,
\end{itemize}then we obtain
\begin{equation}\label{eqn:k=3,r=2}
S_{n,2}(H,J;\balpha) \ll H^{n-5/2}(\log H)^{9}.
\end{equation}

It remains to count vectors $\bnu \in \cS_{n, 2}(H, J; \balpha)$ such that $(\nu_{i_1}, \nu_{i_2}, \nu_{i_3})$ is multiplicatively dependent and $\alpha_j \nu_j = J$ for some $j \in \{i_1, i_2, i_3\}$. There are at most three triples $(\nu_{i_1}, \nu_{i_2}, \nu_{i_3})$ such that $\nu_j \alpha_j = J$ holds for at least two values of $J$, and so there are only $O(H^{n - 3})$ vectors $\bnu \in \cS_{n, 2}(H,J;\balpha)$ such that $\alpha_j \nu_j = J$ for at least two values of $J$. Thus we now count individually across each $j \in \{i_1, i_2, i_3\}$.

Without loss of generality, assume that $j = i_1$. We may let $|\alpha_{i_1}|\mid |J|$. First, if $|J| = |\alpha_{i_1}|$, then $|\nu_{i_1}| = 1$ and so $\bnu$ has rank zero, not two.

Hence, we may let $|\alpha_{i_1}|$ properly divides $|J|$. Let $x = J/\alpha_{i_1} \in \Z$ and $y = -\alpha_{i_2}/\alpha_{i_3} \in \Q$. By swapping $\alpha_{i_2}$ with $\alpha_{i_3}$, we may assume without loss of generality that $|y| > 1$. Since $\nu_{i_1} = x$, the linear condition $\balpha \cdot \bnu$ reduces to $\alpha_{i_2} \nu_{i_2} + \alpha_{i_3} \nu_{i_3} = 0$, or
\[ \nu_{i_3} = (-\alpha_{i_2}/\alpha_{i_3}) \nu_{i_2} = y\nu_{i_2}. \]
We now divide the counting based on whether $|y|$ is a rational power of $|x|$.

First, suppose that $|y|$ is not a rational power of $|x|$. Then by Lemma~\ref{lem:exp}, $(x, \nu_{i_2}, y \nu_{i_2})$ being multiplicatively dependent of rank two implies that there exist nonzero integers $m_1, m_2, m_3 \ll (\log H)^2$, not necessarily positive, such that $x^{m_1} \nu_{i_2}^{m_2} (y \nu_{i_2})^{m_3} = 1$, or equivalently
\[ x^{m_1} y^{m_3} = \nu_{i_2}^{-m_2 - m_3}. \]
Note that $|x|^{m_1} |y|^{m_3} \neq 1$. Thus $|\nu_{i_2}|$ is a rational power of $|x|^{m_1} |y|^{m_3}$, and so there are $O(\log H)$ choices for $\nu_{i_2}$, which implies that there are $O(H^{n - 3} \log H)$ choices for $\bnu$ in this case.

Next, suppose that $|y|$ is a rational power of $|x|$. Then $(x, \nu_{i_2}, y \nu_{i_2})$ is guaranteed to be multiplicatively dependent. Since $\nu$ is of rank two, we have $|y|\ne 1$. Then, $(x, \nu_{i_2}, y \nu_{i_2})$ has rank two if and only if $|\nu_{i_2}|$ is not a rational power of $|x|$. Thus there are
\[ 2|y|^{-1} H + O(\log H) \]
choices of $\nu_{i_2}$, which amounts to
\[ \left(2|y|^{-1} H + O(\log H)\right)(2H + O(\log H))^{n - 3} = 2^{n - 2} |y|^{-1} H + O(H^{n - 3} \log H) \]
choices of $\bnu$, where $(2H + O(\log H))^{n - 3}$ counts the number of choices of the coordinates of $\bnu$ other than $\nu_{i_1}, \nu_{i_2}, \nu_{i_3}$ such that the rank of $\bnu$ is exactly two (i.e. none are $\pm 1$ or powers of other coordinates). 

To conclude our arguments, the case $\alpha_{i_1} \nu_{i_1} = J$ gives a contribution of $2^{n - 2} |y|^{-1} H^{n - 2} + O(H^{n - 3} \log H)$ if $|J/\alpha_{i_1}|$ is an integer greater than $1$, $|\alpha_{i_2}| > |\alpha_{i_3}|$, and $|\alpha_{i_2}/\alpha_{i_3}|$ is a rational power of $|J/\alpha_{i_1}|$ not equal to $1$, and $O(H^{n - 3} \log H)$ otherwise.

Considering contributions from the three cases $\alpha_j \nu_j = J$ across all $j = i_1, i_2, i_3$, we get that
\begin{equation}\label{eqn:k=3,r=2,balpha}
S_{n,2}(H,J;\balpha)= C^{(2);k=3}_{\balpha;J} H^{n-2}+O(H^{n-5/2}(\log H)^{9}),
\end{equation}
    with
\begin{equation}\label{eqn:k=3,r=2,main}
    C^{(2)}_{\alpha;J} = 2^{n - 2} \sum_{j, j', j''} |\alpha_{j'}/\alpha_{j''}|^{-1},
\end{equation}
    where the sum runs across all permutations of $(i_1, i_2, i_3)$ such that the following holds:
\begin{itemize}
    \item   $|J/\alpha_j|$ is a integer greater than $1$, and
    \item   $|\alpha_{j'}/\alpha_{j''}|$ is an integer greater than $1$ that is a rational power of $|J/\alpha_j|$.
\end{itemize}

For the case $r=3$, we consider the related vector $\bnu$ of rank $3$. Arguing similarly as in the case $r=2$ to reduce the number of variables, we obtain a system of equations that consists of~\eqref{eqn:rank3-case1} and~\eqref{eqn:alpha-k=3} with $\{i_1,i_2,i_3\}\subseteq \{j_1,j_2,j_3,j_4\}$. By applying Theorem~\ref{thm:bompil-4var} for each choice of $(k_1,k_2,k_3,k_4)$, there are at most $O(H^{n-5/2}(\log H)^{16})$ possible choices for $\bnu$ in this case. Combining with~\eqref{eqn:rank3-case2-count}, when $r=3$, we obtain
\begin{equation}\label{eqn:k=3,r=3}
S_{n,3}(H,J;\balpha) \ll H^{n-5/2}(\log H)^{16}.\end{equation}
Substituting~\eqref{eqn:k=3,r=2} or~\eqref{eqn:k=3,r=2,balpha} and~\eqref{eqn:k=3,r=3} to~\eqref{eqn:sumrankDorder}, we obtain
\[
S_{n}(H,J;\balpha)=S_{n,0}(H,J;\balpha)+S_{n,1}(H,J;\balpha)+C^{(2);k=3}_{\balpha;J}H^{n-2}+O(H^{n-5/2}(\log H)^{16})
\] and complete the proof of Proposition~\ref{prp:k=2,3,4}.

\section{Multiplicatively dependent vectors of small ranks}\label{sec:small-rank} 
\subsection{Introduction}
The main goal of this section  is to complete the proof of the results in Section~\ref{sec:results} by providing the formulae of the number of related vectors of rank $0$ and $1$. The main argument are provided for $k\ge 4$, which is provided in Section~\ref{sec:k>=4}. The arguments for the case $k=3$ and $k=2$, done in Sections~\ref{sec:k=3-main} and~\ref{sec:k=2-main} respectively, are similar, only that we need to provide some adjustments with regards to the leading coefficient. These arguments complete the proof of Theorem~\ref{thm:main-k>=3} and~\ref{thm:main-k=2}. 

The case $k=1$ is treated separately in Section~\ref{sec:k=1-main}. In particular, this case provides a different asymptotic behaviour for the related vector-counting function, as seen in Theorem~\ref{thm:main-k=1}. The last Section~\ref{sec:positive} discusses the setup where we restrict the vectors $\bnu$ to have strictly positive coordinates, which corresponds to Theorems~\ref{thm:main-positive-integer} and~\ref{thm:main-all-positive-integer}.

Before stating our result, we need to define some vectors related to $\balpha\in \Z^n$. First, for $1\le i\le n$, define  \begin{equation*}
\balpha^*_{i} \coloneqq (\alpha_1,\dots,\alpha_{i-1},\alpha_{i+1},\dots,\alpha_n)\in \Z^{n-1}
\end{equation*} as the vector obtained from removing the $i$-th coordinate in $\balpha$. Next, define
\begin{equation}\label{eqn:alphai1i2} \balpha^-_{i_1,i_2} \coloneqq (\alpha_1,\dots,\alpha_{i_1-1},\alpha_{i_1+1},\dots,\alpha_{i_2-1},\alpha_{i_2+1},\dots,\alpha_n,\alpha_{i_1}-\alpha_{i_2})\in \Z^{n-1} \end{equation}
    as the vector obtained by removing the $i_1$-th and $i_2$-th coordinates of $\balpha$, and appending the resulting vector with $\alpha_{i_1}-\alpha_{i_2}$.
We also define $\balpha^+_{i_1,i_2} $ in a similar way, with $+$ in place of $-$. 
We also denote an indicator function $\delta$ over $\balpha\in \Z^n$ and $J\in \Z$, defined as
\begin{align*}
    \delta_{\balpha}(J)\coloneqq\begin{cases}
        1 & \quad \text{if} \quad \gcd(\balpha)\mid J,\\
        0 & \quad \text{if} \quad \gcd(\balpha)\nmid J.
    \end{cases}
\end{align*}

Related to these objects, we now define
\begin{align}\begin{split}\label{eqn:C0C1}
    & C^{(0)}_{\balpha;J} \coloneqq 2^{n-2}\sum_{i=1}^n [\delta_{\balpha_i^*}(J-\alpha_i)+\delta_{\balpha_i^*}(J+\alpha_i)] \; V_{\balpha_i^*}([-1/2,1/2]^{n-1};0), 
    \\& C^{(1)}_{\balpha;J} \coloneqq 2^{n-2}\sum_{1\le i_1<i_2\le n}\left[\delta_{\balpha^-_{i_1,i_2}}(J) V_{\balpha^-_{i_1,i_2}}([-1/2,1/2]^{n-1};0) + 
    \delta_{\balpha^+_{i_1,i_2}}(J)V_{\balpha^+_{i_1,i_2}}([-1/2,1/2]^{n-1};0)\right], 
    \end{split}
\end{align}
where the notation $V_{\balpha}$ is defined at~\eqref{eq:Valpha-BJ}.

\subsection{The case \texorpdfstring{$k\ge 4$}{k>=4}}\label{sec:k>=4}
In this case, we provide asymptotic formulae for $S_{n,0}(H,J;\balpha)$ and $S_{n,1}(H,J;\balpha)$ in the following propositions.

\begin{proposition}\label{prp:Sn0}
Let $H$ and $J$ be integers, and let $\balpha\in \Z^n$ be a vector with $k$ nonzero coordinates. Suppose that $k\ge 3$.
Then, 
\begin{equation*}
    S_{n, 0}(H, J;\balpha) = C^{(0)}_{\balpha;J}H^{n-2} + 
    \begin{cases}
        O(H^{n-3}+|J|H^{n-3}) & \quad  \text{if} \quad k \le 4, \\
        O(H^{n-3} +|J|^2 H^{n - 4}) & \quad \text{if} \quad k \geq 5.
    \end{cases}
\end{equation*}
\end{proposition}

\begin{proposition}\label{prp:Sn1}
Let $H$ and $J$ be integers, and let $\balpha\in \Z^n$ be a vector with $k$ nonzero coordinates. 
Then,
\begin{equation*}
    S_{n,1}(H,J;\balpha) = C^{(1)}_{\balpha;J} H^{n - 2}+ \begin{cases}
        O( H^{n-5/2}+|J|H^{n-3} ) & \quad \text{if} \quad k = 4 \text{ and } J\ne 0, \\
        O(H^{n - 5/2} + |J|^2 H^{n - 4}) & \quad \text{if} \quad k \ge 5.
    \end{cases}
\end{equation*}
\end{proposition}
Combining Propositions~\ref{prp:Sn0},~\ref{prp:Sn1} and either Corollary~\ref{cor:mainterm} (for $k\ge 5$) or Proposition~\ref{prp:k=2,3,4} (for $k=4$) completes the proof of Theorem~\ref{thm:main-k>=3} for $k\ge 4$, where we define 
\begin{equation}\label{eqn:Calpha}
    C_{\balpha;J} \coloneqq C^{(0)}_{\balpha;J} + C^{(1)}_{\balpha;J}
\end{equation} and the constants in~\eqref{eqn:C0C1}.
The remainder of this section is dedicated to providing the proofs of these propositions.
\begin{proof}[Proof of Proposition~\ref{prp:Sn0}]
To count $S_{n,0}(H,J;\balpha)$ for $n \ge 3$, we only need to count the number of vectors that has $\pm1$ as one of its coordinates and lies on the hyperplane $\balpha\cdot \bnu=J$.
The number of vectors $\bnu \in \cS_{n,0}(H,J;\balpha)$ with at least two coordinates equal to $\pm 1$ is
\begin{equation}\label{eqn:Sn0-twodistinct}
    O(H^{n - 3})
\end{equation}
due to the hyperplane equation, since $k \geq 3$.

Hence, we now count the number of vectors  $\bnu \in\cS_{n,0}(H,J;\balpha)$   such that $\nu_{i}=\pm 1$ for exactly one index $i \le n$. For a fixed $i$,  the number of vectors $\bnu$ such that $\nu_{i} = \pm1$ and $\balpha \cdot \bnu = J$ is 
\begin{align}\label{eqn:alpha-star}\begin{split}
    &\#\left\{\bnu \in (\Z \cap [-H, H])^{n - 1} : \balpha^*_{i} \cdot \bnu = J - \alpha_{i}\right\} \\
    &\qquad+\#\left\{\bnu \in (\Z \cap [-H, H])^{n - 1} : \balpha^*_{i} \cdot \bnu = J + \alpha_{i}\right\} 
\end{split}
\end{align}
When these quantities are nonzero,   we may use Theorem~\ref{thm:shifted-lattice-box-count} and Corollary~\ref{cor:Valpha-BJ-approx} to conclude that~\eqref{eqn:alpha-star} is equal to 
\begin{equation*}\begin{split}
     &V_{\balpha_i^*}([-H,H]^{n-1};J-\alpha_i)+  V_{\balpha_i^*}([-H,H]^{n-1};J+\alpha_i)+O(H^{n-3})\\
     &\quad= 2^{n - 1}V_{\balpha_i^*}([-1/2,1/2]^{n-1};0) H^{n-2}+\begin{cases}
        O(H^{n-3}+|J|H^{n-3}) & \quad \text{if} \quad k \le 4, \\
        O(H^{n-3} +|J|^2 H^{n - 4}) & \quad \text{if} \quad k \geq 5.
\end{cases}
\end{split}
\end{equation*}
When at least one of the terms of~\eqref{eqn:alpha-star} is zero, we adjust accordingly. Then, we add~\eqref{eqn:alpha-star} over all possible $i$ and combine the result with~\eqref{eqn:Sn0-twodistinct} to complete the proof of Proposition~\ref{prp:Sn0}.
\end{proof}

\begin{proof}[Proof of Proposition~\ref{prp:Sn1}] Let $\bnu \in \cS_{n,1}(H,J;\balpha)$. Then, there exist a pair of coordinates $(\nu_{i_1},\nu_{i_2})$ and nonzero integers $k_1, k_2$ such that \begin{align}\label{eqn:rank1}
    \nu_{i_1}^{k_1}\nu_{i_2}^{k_2}=1.
\end{align} 

We first bound the number of vectors $\bnu$ associated with two distinct such pairs $(i_1,i_2)$ and $(i_3,i_4)$, by showing that there are at most
\begin{equation}\label{eqn:twodistinct}
    O(H^{n-3}(\log H)^6).
\end{equation}
such vectors $\bnu$. This vector would satisfy
\begin{equation}\label{eqn:rank1-intersect}
    \nu_{i_1}^{k_1} = \nu_{i_2}^{k_2} \quad \text{and} \quad \nu_{i_3}^{k_3} = \nu_{i_4}^{k_4}
\end{equation}
for some positive integers $k_1, k_2, k_3, k_4 \ll \log H$, by Lemma~\ref{lem:exp}. We now fix these variables and show that the number of possible choices for $\bnu$ satisfying~\eqref{eqn:rank1-intersect} is $O(H^{n - 3} (\log H)^2)$, baseed on the structure of $\balpha$. 

First, suppose there exists an index $i \notin \{i_1, i_2, i_3, i_4\}$ with $\alpha_i = 0$. Suppose first that $i_1, i_2, i_3, i_4$ are pairwise distinct. Then $\balpha \cdot \bnu = J$ has $O(H^{n - 5})$ solutions after fixing $\nu_{i_1}, \nu_{i_2}, \nu_{i_3}, \nu_{i_4}$. For $k_1, k_2, k_3, k_4$ fixed,~\eqref{eqn:rank1-intersect} has $O(H^2)$ solutions, thus giving us $O(H^{n - 3})$ vectors satisfying~\eqref{eqn:rank1-intersect}. On the other hand, in the case where $i_{j_1} = i_{j_2}$ for some $j_1 \in \{1, 2\}$ and $j_2 \in \{3, 4\}$, the same argument holds, with $O(H^{n - 4})$ solutions to $\balpha \cdot \bnu = J$ and $O(H)$ solutions to~\eqref{eqn:rank1-intersect}. In particular,~\eqref{eqn:twodistinct} holds whenever $k \geq 5$.

Now suppose that $\alpha_i = 0$ for all $i \notin \{i_1, i_2, i_3, i_4\}$. In particular, $k \leq 4$, and we assume $J \neq 0$. First consider the case where $i_1, i_2, i_3, i_4$ are pairwise distinct. From~\eqref{eqn:rank1-intersect}, we have
\[ (\nu_{i_1}, \nu_{i_2}, \nu_{i_3}, \nu_{i_4}) = (\pm A^{k_2/a}, \pm A^{k_1/a}, \pm B^{k_4/b}, \pm B^{k_3/b}) \]
for some integers $A$ and $B$, where $a = \gcd(k_1, k_2)$ and $b = \gcd(k_3, k_4)$. The equation
\[ \pm \alpha_{i_1} A^{k_2/a} \pm \alpha_{i_2} A^{k_1/a} \pm \alpha_{i_3} B^{k_4/b} \pm \alpha_{i_4} B^{k_3/b} = J \]
has $O(H \max\{k_1, k_2, k_3, k_4\}) = O(H \log H)$ solutions in terms of the pairs $(A, B)$ for any fixed choices of the signs, except possibly when the left hand side is constant as a polynomial in $A$ and $B$. However, in that case the constant is $0 \neq J$, so the equation has no solutions. Together with $O(H^{n - 4})$ choices for the other coordinates, we get $O(H^{n - 3} \log H)$ vectors $\bnu$ satisfying~\eqref{eqn:rank1-intersect}.

Now consider the case where some of the indices $i_1, i_2, i_3, i_4$ are equal. Without loss of generality, we assume $i_1 = i_4$. Then 
\[ (\nu_{i_1}, \nu_{i_2}, \nu_{i_3}) = (\pm A^{k_2 k_3/c}, \pm A^{k_1 k_3/c}, \pm A^{k_2 k_4/c}) \]
for some integer $A$, where $c = \gcd(k_2 k_3, k_1 k_3, k_1 k_4)$. With similar argument as in the previous case, we may show that
\[ \pm \alpha_{i_1} A^{k_2 k_3/c} \pm \alpha_{i_2} A^{k_1 k_3/c} \pm \alpha_{i_3} A^{k_2 k_4/c} = J \]
has $O(\max\{k_1, k_2, k_3, k_4\}^2) = O((\log H)^2)$ solutions in terms of $A$ for any fixed choices of the signs. Together with $O(H^{n - 3})$ choices for the other coordinates, Thus we get $O(H^{n - 3} (\log H)^2)$ vectors $\bnu$ satisfying~\eqref{eqn:rank1-intersect}. This proves the bound~\eqref{eqn:twodistinct}.

Therefore, we may assume $\bnu$ has exactly one pair of indices satistfying~\eqref{eqn:rank1}. We first consider the case $(k_{i_1},k_{i_2})=(t,\pm t)$, which implies
\[ \nu_{i_1} = \pm \nu_{i_2}. \]

First, suppose that $\nu_{i_1} = -\nu_{i_2}$. Note that since we count over $[-H,H]^{n-1}$, we may assume $i_1<i_2$.
In this case, we count the number of solutions to
\begin{equation}\label{eqn:linear-}
    \alpha_1 \nu_1 +\dots + (\alpha_{i_1}-\alpha_{i_2})\nu_{i_1} +\dots + \alpha_n\nu_n=J,
\end{equation}
with $|\nu_i|\le H$. Note that~\eqref{eqn:linear-} has at least $k-1$ and $k-2$ terms when $\alpha_{i_1}\ne \alpha_{i_2}$ and $\alpha_{i_1} = \alpha_{i_2}$, respectively. Since $k \geq 4$, there are $O(H^{n - 3})$ vectors satisfying~\eqref{eqn:linear-} with one of its coordinates equal to $\pm 1$.

We now count the number of integer solutions of~\eqref{eqn:linear-} with $|\nu_i|\le H$. When $\gcd(\balpha_{i_1,i_2}^-)\mid J$, by using notations of~\eqref{eqn:alphai1i2}, Theorem~\ref{thm:shifted-lattice-box-count} and Corollary~\ref{cor:Valpha-BJ-approx}, this is exactly
\begin{align}\label{eqn:tpmt-int}\begin{split}
&  V_{\balpha^-_{i_1,i_2}} ([-H,H]^{n-1};J)+O(H^{n-3})\\
=& 2^{n-2} V_{\balpha^-_{i_1,i_2}}([-1/2,1/2]^{n-1};0)H^{n-2} + 
    \begin{cases}
        O(H^{n - 3} + |J| H^{n - 3}) & \quad \text{if} \quad k\le   4, \\
        O(H^{n - 3} + |J|^2 H^{n - 4}) & \quad \text{if} \quad k \ge 5.
    \end{cases}\end{split}
\end{align}

Next, suppose that $\nu_{i_1} = \nu_{i_2}$. Similar to the previous case, 
for fixed $(i_1, i_2)$,  we now count the number of solutions to
\begin{equation}\label{eqn:linear+}
    \alpha_1 \nu_1 +\dots + (\alpha_{i_1}+\alpha_{i_2})\nu_{i_1} +\dots + \alpha_n\nu_n=J,
\end{equation}
where we consider this as an equation of $n-1$ integer variables, all with $|\nu_i|\le H$. 
Applying similar arguments as in the previous case, if this equation has integer solutions, then the number of solutions is
\begin{equation}\label{eqn:tpmt-int2}
2^{n-2}V_{\balpha^+_{i_1,i_2}}([-1/2,1/2]^{n-1};0)H^{n-2} + 
    \begin{cases}
        O(H^{n - 3} + |J| H^{n - 3}) & \quad \text{if} \quad k\le 4, \\
        O(H^{n - 3} + |J|^2 H^{n - 4}) & \quad \text{if} \quad k \ge 5.
    \end{cases}
\end{equation}

Adding~\eqref{eqn:tpmt-int} and~\eqref{eqn:tpmt-int2} over all possible $i_1< i_2$, we obtain the number of such vectors in this case $(k_{i_1},k_{i_2})=(t,\pm t)$ is
\begin{equation}\label{eqn:LnK1integer}
    C^{(1)}_{\balpha;J}H^{n-2} 
+\begin{cases}
        O(H^{n - 3} + |J| H^{n - 3}) & \quad \text{if} \quad k\le 4, \\
        O(H^{n - 3} + |J|^2 H^{n - 4}) & \quad \text{if} \quad k \ge 5.
    \end{cases}
\end{equation}

We now consider the remaining cases of the exponents of~\eqref{eqn:rank1},  which is the case $(k_{i_1},k_{i_2}) \ne (t,\pm t)$.
Without loss of generality and using Lemma~\ref{lem:exp}, 
we assume \begin{align*}
    \log H \gg k_1 > 0 > k_2, \quad |k_1|>|k_2|, \quad \nu_{i_1},\nu_{i_2}>0 \quad \text{and} \quad \gcd(k_1,k_2)=1.
\end{align*}
Denote $(t_1,t_2)\coloneqq (k_1,-k_2)$. Therefore,~\eqref{eqn:rank1} is equivalent to \begin{equation*}
   \nu_{i_1}^{t_1}=\nu_{i_2}^{t_2}. 
\end{equation*}
We have that $\nu_{i_2}$ is a $t_1$-th perfect power, and $\nu_{i_1}$ is uniquely determined. Therefore, the number of choices of $(\nu_{i_1}, \nu_{i_2})$ is
\[ \sum_{t_2 < t_1 \ll \log H} O(H^{1/t_1}) = O(H^{1/2}). \]

Counting the other coordinates of $\bnu$ with similar methods as in the other case, we have that the number of vectors $\bnu$ is at most 
\begin{equation}\label{eqn:rank1-wildcases}
   O(H^{n-3}) O(H^{1/2})=O(H^{n-5/2}).
\end{equation} when $k=1$ or $k\ge 4$. Adding~\eqref{eqn:LnK1integer},~\eqref{eqn:rank1-wildcases} and~\eqref{eqn:twodistinct} completes our proof. \end{proof}

\subsection{The case \texorpdfstring{$k=3$}{k=3}}\label{sec:k=3-main} In this section, we compute the main term of $S_{n}(H,J;\balpha)$ when $k=3$ and $J\ne 0$. Recalling Proposition~\ref{prp:k=2,3,4}, we have \[
S_{n}(H,J;\balpha) = S_{n,0}(H,J;\balpha)+S_{n,1}(H,J;\balpha)+C^{(2);k=3}_{\balpha;J}H^{n-2}+O(H^{n-5/2}(\log H)^{16}),
\]  where $C^{(2);k=3}_{\balpha;J}$ is defined in~\eqref{eqn:k=3,r=2,main}.
 By directly applying Proposition~\ref{prp:Sn0}, we obtain $S_{n,0}(H,J;\balpha)$.  Therefore, it remains to compute $S_{n,1}(H,J;\balpha)$. We employ an argument analogous to that in the proof of Proposition~\ref{prp:Sn1}, however we add a correcting factor for some specific choices of $\balpha$ and $J$ which we define shortly.

Let $\alpha_{i_1}$, $\alpha_{i_2}$, and $\alpha_{i_3}$ be the nonzero coordinates of $\balpha$. Define the quantity
\[ C^{(1);k=3}_{\balpha;J} = C^{(1)}_{\balpha;J} - 2^{n - 2} \#X, \]
where $X \subseteq \{i_1, i_2, i_3\}$ is the set of indices $j$ such that $|\alpha_j| = |J|$ and $|\alpha_{j'}| = |\alpha_{j''}|$, where $j', j''$ are the indices in $\{i_1, i_2, i_3\}$ other than $j$. For example, when $(\alpha_{i_1},\alpha_{i_2},\alpha_{i_3})=(\pm1,\pm1,\pm J)$ with $J\ne 1$, $\#X=1$. However, when $J=1$, $\#X=3$. Note that $\#X \in \{0, 1, 3\}$, with $\#X = 3$ if and only if all the nonzero coordinates of $\balpha$ are $\pm J$.

We are now ready to state our result for $k=3$.

\begin{proposition}\label{prp:Sn1-k=3}
Let $H$ and $J \neq 0$ be integers, and let $\balpha \in \Z^n$ be a vector with three nonzero coordinates. Then,
\[ S_{n,1}(H,J;\balpha) = C^{(1);k=3}_{\balpha;J} H^{n - 2} + O(H^{n-5/2}). \] In particular, as $H,|J|\to \infty$, \[
S_{n,1}(H,J;\balpha)= C^{(1)}_{\balpha;J} H^{n - 2} + O(H^{n-5/2}).
\]
\end{proposition}

As a result, by combining Proposition~\ref{prp:Sn0}, Proposition~\ref{prp:Sn1-k=3} and Proposition~\ref{prp:k=2,3,4}, we obtain Theorem~\ref{thm:main-k>=3} for $k = 3$, with
\begin{equation}\label{eqn:Calpha-k=3}
C_{\balpha;J} := C^{(0)}_{\balpha;J} + C^{(1);k=3}_{\balpha;J} + C^{(2);k=3}_{\balpha;J}.
\end{equation}
   
We now prove Proposition~\ref{prp:Sn1-k=3}. Let the related plane equation be \[
\alpha_{j_1}\nu_{j_1}+\alpha_{j_2}\nu_{j_2}+\alpha_{j_3}\nu_{j_3}=J.\]

With similar arguments as in the proof of Proposition~\ref{prp:Sn1}, we only need to consider the number of pairs $(\nu_{i_1},\nu_{i_2})$ of rank $1$. Note that the arguments that gives the error bound $O(H^{n-3}(\log H)^4)$ and $O(H^{n-5/2})$ from~\eqref{eqn:rank1-intersect} and~\eqref{eqn:rank1-wildcases}, respectively, still holds in this case. Therefore, we only need to count solutions to~\eqref{eqn:linear-} and~\eqref{eqn:linear+} across all choices of indices $i_1 < i_2$. Note that when $\{i_1,i_2\} \not\subseteq \{j_1,j_2,j_3\}$, we may follow a similar argument as in the referred equations. 

However, a different scenario happens when, without loss of generality, $i_1=j_1$ and $i_2=j_2$. If $\alpha_{j_3} \ne \pm J$, then we may proceed with the arguments in~\eqref{eqn:linear-} and~\eqref{eqn:linear+} to obtain Proposition~\ref{prp:Sn1-k=3}, with $\#X=0$. However, when $\alpha_{j_3}=\pm J$, then $\nu_{j_3}=\pm 1$ and the corresponding vector $\bnu$ would have rank zero. Therefore, we need to exclude these family of vectors from the counting of $S_{n,1}(H,J;\balpha)$.

We now compute the number of such vectors. Suppose that $\alpha_{i_1} = \alpha_{i_2} \neq 0$, and $\alpha_{i_3} = \pm J$. By applying Theorem~\ref{thm:shifted-lattice-box-count}, the number of vectors satisfying~\eqref{eqn:linear-} is
\[ 2^{n - 2} V_{J\bfe_1}([-1/2, 1/2]^{n-1};J) H^{n - 2} + O(H^{n-3}) = 2^{n-2} H^{n-2} + O(H^{n-3}). \]
The same argument also holds for the case $\alpha_{i_1}=-\alpha_{i_2}$.

As a result, by considering the other permutations of $(i_1,i_2)$, we obtain
\[ S_{n,1}(H,J;\balpha) = (C^{(1)}_{\balpha;J} - 2^{n-2} \#X) H^{n-2} + O(H^{n-5/2}), \]
where $X$ is defined before Proposition~\ref{prp:Sn1-k=3}. This completes the proof of Proposition~\ref{prp:Sn1-k=3}.

\subsection{The case \texorpdfstring{$k=2$}{k=2}}\label{sec:k=2-main}

In this case, we consider the hyperplane \begin{equation}\label{eqn:multdep-k=2}
\alpha_{j_1}\nu_{j_1}+\alpha_{j_2}\nu_{j_2}=J
\end{equation}
with $J\ne 0$. Also, define $\cS'_2(J;\alpha_{j_1},\alpha_{j_2})$ as the set of multiplicatively dependent integer solutions of~\eqref{eqn:multdep-k=2} with $|\nu_{j_1}|,|\nu_{j_2}| > 1$ and $|\nu_{j_1}| \ne |\nu_{j_2}|$. Then, let
\[ S'_2(J;\alpha_{j_1},\alpha_{j_2})=\# \cS'_2(J;\alpha_{j_1},\alpha_{j_2}). \]
Note that $S'_2(J;\alpha_{j_1},\alpha_{j_2})$ is well-defined from Lemma~\ref{lem:extra-count-k=2}. Also, define the quantities
\[ C^{(0);k=2}_{\balpha; J} = \begin{cases} 
     C^{(0)}_{\balpha; J} - 2^{n - 2} & \quad \text{if} \quad J \in \{\pm(\alpha_{j_1} + \alpha_{j_2}), \pm(\alpha_{j_1} - \alpha_{j_2})\},\\
    C^{(0)}_{\balpha; J} & \quad \text{otherwise},
    \end{cases} \]
    and
\[ C^{(1);k=2}_{\balpha; J} = \begin{cases} 
    C^{(1)}_{\balpha; J} + 2^{n-2} S'_2(J;\alpha_{j_1},\alpha_{j_2}) - 2^{n-2} & \quad \text{if} \quad J \in \{\pm(\alpha_{j_1} + \alpha_{j_2}), \pm(\alpha_{j_1} - \alpha_{j_2})\}, \\
    C^{(1)}_{\balpha; J} + 2^{n-2} S'_2(J;\alpha_{j_1},\alpha_{j_2}) & \quad \text{otherwise.}
\end{cases} \]
We are now able to state our proposition for $k=2$.

\begin{proposition}\label{prp:Sn0,1-k=2}
Let $H$ and $J \neq 0$ be integers, and $\balpha \in \Z^n$ be a vector with two nonzero coordinates, Then, as $H \to \infty$,
\begin{align*} &S_{n,0}(H,J;\balpha) = C^{(0);k=2}_{\balpha; J} H^{n - 2} + O(H^{n - 3}), \\&S_{n,1}(H,J;\balpha) = C^{(1);k=2}_{\balpha; J} H^{n - 2} +  O(H^{n - 5/2}). 
\end{align*}
In particular, as $|J|,H\to \infty$, we have \begin{align*} &S_{n,0}(H,J;\balpha) = C^{(0)}_{\balpha; J} H^{n - 2} + O(H^{n - 3}), \\&S_{n,1}(H,J;\balpha) = (C^{(1)}_{\balpha; J}+2^{n-2} S'_2(J;\alpha_{j_1},\alpha_{j_2}))H^{n - 2} +  O(H^{n - 5/2}). 
\end{align*}
\end{proposition} 

Note that our proof of Theorem~\ref{thm:main-k=2} is complete by combining Propositions~\ref{prp:Sn0,1-k=2} and~\ref{prp:k=2,3,4}. The rest of this section are the arguments for the proof of Proposition~\ref{prp:Sn0,1-k=2}. In addition, we also note that finding a result for a fixed $J$ and $H$ is harder in this case, due to reasons outlined in Section~\ref{sec:strength}.

We start by proving Proposition~\ref{prp:Sn0,1-k=2} for $r=0$. We note that the exact argument from  the proof of Proposition~\ref{prp:Sn0} works when $J \notin \{\pm(\alpha_{j_1} + \alpha_{j_2}), \pm(\alpha_{j_1} - \alpha_{j_2})\}$. Otherwise, we may write
\begin{equation}\label{eqn:Jpm1} J = \epsilon_1 \alpha_{j_1} + \epsilon_2 \alpha_{j_2} \end{equation}
for some $\epsilon_1, \epsilon_2 \in \{-1, 1\}$; the pair $(\epsilon_1, \epsilon_2)$ is unique since $J, \alpha_{j_1}, \alpha_{j_2} \neq 0$.

We may follow the proof of Proposition~\ref{prp:Sn0} in this case, except when counting the number of vectors $\bnu\in  \cS_{n,0}(H,J;\balpha)$ with at least two coordinates equal to $\pm 1$. In particular, note that there are $O(H^{n-3})$ such vector  $\bnu$  with $(\nu_{j_1}, \nu_{j_2}) \neq (\epsilon_1, \epsilon_2)$. However, note that when $(\nu_{j_1}, \nu_{j_2}) = (\epsilon_1, \epsilon_2)$, which is possible from~\eqref{eqn:Jpm1}, we obtain $(2H)^{n-2}+O(H^{n-3})$ such vectors. Therefore, the number of such vectors overall is
\[ (2H)^{n-2}+O(H^{n-3}). \]

Hence, when $J \in \{\pm(\alpha_{j_1} + \alpha_{j_2}), \pm(\alpha_{j_1} - \alpha_{j_2})\}$, we obtain
\[ S_{n,1}(H,J;\balpha) = (C^{(0)}_{\balpha; J} - 2^{n - 2}) H^{n - 2} + O(H^{n - 3}). \]
This completes the proof of Proposition~\ref{prp:Sn0,1-k=2} for $r=0$.

We now proceed to prove Proposition~\ref{prp:Sn0,1-k=2} for $r=1$. As in the proof of Proposition~\ref{prp:Sn1}, there $O(H^{n-3}(\log H)^6)$ vectors $\bnu$ such that the related rank equations hold for at least two distinct pairs of coordinates. We now proceed to count solutions of~\eqref{eqn:linear-} and~\eqref{eqn:linear+} in this case.

We first count the number of solutions of~\eqref{eqn:linear-} of rank zero. Note that when $\{i_1, i_2\} \neq \{j_1, j_2\}$, $\balpha_{i_1,i_2}^-$  has two nonzero coordinates. Therefore, following the proof of Proposition~\ref{prp:Sn1}, in this case the equation~\eqref{eqn:linear-} has $O(H^{n-3})$ solutions with one coordinate equal to $\pm 1$. Now, when $\{i_1, i_2\} = \{j_1, j_2\}$, there are also $O(H^{n-3})$ vectors of this kind when $|\alpha_{j_1} - \alpha_{j_2}| \ne |J|$. 

The remaining case is when $|\alpha_{j_1} - \alpha_{j_2}| = |J|$, where every vector $\bnu$ satisfying~\eqref{eqn:linear-} satisfies $\nu_{j_1} = \pm 1$. By choosing the other $n - 2$ coordinates arbitrarily, the number of such vector is $(2H)^{n-2}+O(H^{n-3})$. Similar arguments hold for~\eqref{eqn:linear+}, where instead we consider $|\alpha_{j_1} + \alpha_{j_2}| = |J|$.

Since $J, \alpha_{j_1}, \alpha_{j_2} \neq 0$, only one of the two equations $|\alpha_{j_1} - \alpha_{j_2}| = |J|$ and $|\alpha_{j_1} + \alpha_{j_2}| = |J|$ could hold at one time. Therefore, we only need to exclude exactly $(2H)^{n-2}+O(H^{n-3})$ such vector in the case one of this scenario happens. We then obtain \begin{equation}\label{eqn:Sn1-k=2-identical}
C^{(1)}_{\balpha;J}H^{n-2}+O(H^{n-3})-\begin{cases}
 (2H)^{n-2} &\quad \text{if} \quad J\in \{\pm(\alpha_{j_1} + \alpha_{j_2}), \pm(\alpha_{j_1} - \alpha_{j_2})\},\\
 0 &\quad \text{otherwise}.
\end{cases}
\end{equation}
vectors of rank $1$ with $|\nu_{j_1}| = |\nu_{j_2}|$ in this case.

It remains to deal with the case where $\nu_{i_1}^{t_1} = \nu_{i_2}^{t_2}$ for some positive integers $t_1 \neq t_2$. As in the proof of Proposition~\ref{prp:Sn1}, there are $O(H^{1/2})$ choices for the pair $(\nu_{i_1}, \nu_{i_2})$, and when $\{i_1, i_2\} \neq \{j_1, j_2\}$, they correspond to $O(H^{n-5/2})$ vectors $\bnu \in \cS_{n,1}(H,J;\balpha)$. 

However, when $\{i_1, i_2\} = \{j_1, j_2\}$, the only pairs $(\nu_{j_1}, \nu_{j_2})$ corresponding to vectors $\bnu \in \cS_{n,1}(H;J;\balpha)$ are those satisfying~\eqref{eqn:multdep-k=2} with $|\nu_{j_1}|,|\nu_{j_2}| > 1$ and $|\nu_{j_1}|\ne |\nu_{j_2}|$, i.e. $(\nu_{j_1}, \nu_{j_2}) \in \cS'_2(J;\alpha_{j_1},\alpha_{j_2})$. Each such pair corresponds to $(2H - 2)^{n - 2} = 2^{n - 2} H^{n - 2} + O(H^{n - 3})$ vectors $\bnu \in \cS_{n,1}(H;J;\balpha)$. Thus the number of $\bnu \in \cS_{n,1}(H;J;\balpha)$ such that $\nu_{i_1}^{t_1} = \nu_{i_2}^{t_2}$ for some positive integers $t_1 \neq t_2$ is
\begin{equation}\label{eqn:Sn1-k=2-wildcases} 2^{n - 2} S'_2(J;\alpha_{j_1},\alpha_{j_2}) + O(H^{n-3}). \end{equation}

Combining~\eqref{eqn:Sn1-k=2-identical} and~\eqref{eqn:Sn1-k=2-wildcases} completes the proof of Proposition~\ref{prp:Sn0,1-k=2} for $r=1$.

\subsection{The case \texorpdfstring{$k=1$}{k=1}}\label{sec:k=1-main}
As mentioned in Section~\ref{sec:k=1}, to count $S_{n,0}(H,J;\bfe_1)$ and $S_{n,1}(H,J;\bfe_1)$ in this case, we may assume $|J| > 1$. Note that $S_{n,0}(H,J;\bfe_1)$ is the number of vectors $\bnu\in \cS_{n}(H,J;\bfe_1)$ such that $\nu_1=J$ and at least one of the other coordinates of $\bnu$ is equal to $\pm 1$. Picking the coordinate location and applying the principles of inclusion and exclusion, we obtain \begin{align}\label{eqn:k=1,r=0}\begin{split}
S_{n,0}(H,J;\bfe_1)&= 2(n-1)(2H)^{n-2}+O(H^{n-3})\\&=2^{n-1}(n-1)H^{n-2}+O(H^{n-3}).
\end{split}\end{align}

We now compute $S_{n,1}(H,J;\bfe_1)$. As in Section~\ref{sec:k=1}, we may suppose $J$ is positive and not a perfect power. There are two possible scenarios for  a vector $\bnu \in \cS_{n,1}(H,J;\bfe_1)$; either \begin{itemize}
    \item $\nu_i=\pm J^k$ for some positive integer $k\le \log H/\log J$, or 
    \item There exists a pair of indices $(i_1, i_2)$ with $1 < i_1 < i_2$ and $k_1,k_2\ll \log H$ (from Lemma~\ref{lem:exp}) with \[
    \nu_{i_1}^{k_1} = \nu_{i_2}^{k_2}.
    \]
\end{itemize}
Note that trivially there are $O(H^{n-3} (\log H)^4)$ cases where these scenarios happened twice. Therefore, we may count each cases separately.

For the first case, there are $n-1$ possible choices for $i$, two possible sign choices and $\lfloor \log H/\log J\rfloor$ choices for $k$. Therefore,  we obtain \begin{align}\label{eqn:k=1-J}
&2(n-1) \left\lfloor \dfrac{\log H}{\log J}\right\rfloor (2H)^{n-2} + O(H^{n-3}) = 2^{n-1}(n-1) \left\lfloor \dfrac{\log H}{\log J}\right\rfloor H^{n-2} +O(H^{n-3})
\end{align} vectors $\bnu \in \cS_{n,1}(H,J;\bfe_1)$ in this case.

For the second case, when $|k_1|=|k_2|$ we obtain $|\nu_{i_1}|=|\nu_{i_2}|$. Without loss of generality, we may let $\nu_{i_1}=\nu_{i_2}=A$ for some integer $1<A\le H$. There are $H-1=H+O(1)$ choices of $A$ and $(n-1)(n-2)/2$ choices for $(i_1,i_2)$. Considering the four sign choices of $\nu_{i_1}$, $\nu_{i_2}$, we obtain \begin{equation}\label{eqn:k=1-eq}
2(n-1)(n-2)H(2H)^{n-3}+O(H^{n-3})=2^{n-2}(n-1)(n-2)H^{n-2}+O(H^{n-3}).
\end{equation} vectors $\bnu \in \cS_{n,1}(H,J;\bfe_1)$ in this case.

For the case $|k_1|\ne |k_2|$, we apply similar argument as in~\eqref{eqn:rank1-wildcases} to obtain at most \begin{equation}\label{eqn:k=1-wildcases}
O(H^{n-5/2})
\end{equation} vectors $\bnu \in \cS_{n,1}(H,J;\bfe_1)$ in this case.

Adding~\eqref{eqn:k=1-J},~\eqref{eqn:k=1-eq} and~\eqref{eqn:k=1-wildcases} together, we obtain 
\begin{equation}\label{eqn:k=1,r=1}
    S_{n,1}(H,J;\bfe_1)=2^{n-2}(n-1)\left(n-2+2\left \lfloor \dfrac{\log H}{\log J}\right\rfloor\right)H^{n-2} + O(H^{n-5/2}).
\end{equation}
Adding~\eqref{eqn:k=1,r=0} and~\eqref{eqn:k=1,r=1} to Proposition~\ref{prp:k=1}, we obtain 
\[ S_{n}(H,J;\bfe_1)= C_{\bfe_1;J}(H)H^{n-2}+O(H^{n-5/2}), \] with
\begin{align}\label{eqn:Cbfe1}\begin{split}
C_{\bfe_1;J}(H) &\coloneqq 2^{n-1}(n-1) +2^{n-2}(n-1)\left(n-2+2\left \lfloor \dfrac{\log H}{\log J}\right\rfloor\right)+\dfrac{2^{n-1}(n-1)(n-2)}{J-1}
\\&=2^{n-2}(n-1) \left( n+2\left \lfloor \dfrac{\log H}{\log J}\right\rfloor +  \dfrac{2(n-2)}{J-1}\right).
\end{split}\end{align}
This completes the proof of Theorem~\ref{thm:main-k=1}.

Finally, note that the above formula for $S_{n}(H,J;\bfe_1)$ holds when $J$ is positive and not a prime power. In general, as long as $|J| > 1$, we can replace the occurrences of $J$ in~\eqref{eqn:Cbfe1} with $f(|J|)$, the smallest positive integer $A$ such that $|J|$ is a power of $A$ (as defined in Theorem~\ref{thm:main-k=1}).

\subsection{Multiplicatively dependent positive integer vectors}\label{sec:positive}

In this section, we compute $S_n^+(H, J; \balpha)$ for all $\balpha$. First, we outline the differences of counting $S_{n,r}^+(H,J;\balpha)$ compared to  $S_{n,r}(H,J;\balpha)$, which stems from counting solutions of the related hyperplane equation $\balpha \cdot \bnu =J$. In particular, note that our main tools in Section~\ref{sec:count-vol} are worded differently in the case where we limit ourselves to positive integer vectors. Neverthless, most of our arguments in the previous section still hold in this setup.

First, for the case $k=1$, we note that by taking the absolute values of each of the coordinates, each vector $\cS_n^+(H,J;\bfe_1)$ corresponds to exactly $2^{n-1}$ vectors in $\cS_n(H,J;\bfe_1)$. We then use Theorem~\ref{thm:main-k=1} to complete our proof.

Next, we proceed to the case $k\ge 2$. For technical reasons, we only discuss the case where $\balpha$ has either at least two positive and two negative coordinates (which, in particular, implies $k \geq 4$) and the case where all coordinates of $\balpha$ are positive. 

Returning to our problem, we count the number of vectors in $\cS_n^+(H, J; \balpha)$ based on their multiplicative rank. We begin similarly as our arguments for $S_n(H,J;\balpha)$, by bounding the number of vectors based on its multiplicative rank. We denote $S_{n,r}^+(H,J;\balpha)$ as the number of vectors with multiplicative rank $r$. Since trivially
\[ S_{n,r}^+(H,J;\balpha) \leq S_{n,r}(H,J;\balpha), \]
the same upper bounds previously derived for $S_{n,r}(H,J;\balpha)$ also hold for  $S_{n,r}^+(H,J;\balpha)$. In the two cases that we are considering, we have $S_{n,r}(H,J;\balpha) = o(H^{n - 2})$ for $r \ge 2$, except when all entries of $\balpha$ are positive and $k = n = 3$. In this case, while $S_{3,2}(H,J;\balpha)$ contributes to the main term, all such contributions come from the vectors $\bnu$ satisfying~\eqref{eqn:alpha-k=3} with $\alpha_j \nu_j = J$ for some $j \in \{i_1, i_2, i_3\}$. However,~\eqref{eqn:alpha-k=3} implies that no such vectors $\bnu$ have all entries positive, since $\alpha_{i_1}, \alpha_{i_2}, \alpha_{i_3} > 0$. Removing these vectors, we obtain
\[ S_{3,2}^+(H,J;\balpha) = O(H^{n - 5/2} (\log H)^{9}). \]
Therefore, we only need to compute $S_{n, 0}^+(H, J; \balpha)$ and $S_{n, 1}^+(H, J; \balpha)$, which is done after we introduce the following quantities:
\begin{align*}
   &C^{+,(0)}_{\balpha;J} \coloneqq \sum_{i=1}^n \delta_{\balpha_i^*}(J-\alpha_i)V_{\balpha_i^*}([0,1]^{n-1};0),\\
   &C^{+,(1)}_{\balpha;J}\coloneqq \sum_{1\le i_1<i_2\le n} \delta_{\balpha^+_{i_1,i_2}}(J)V_{\balpha^+_{i_1,i_2}}([0,1]^{n-1};0),
\end{align*}
Note the difference between $C^{+,(r)}_{\balpha;J}$ and $C^{(r)}_{\balpha;J}$, which is due to the fact that we fix the sign of the coordinates of $\bnu$ here. Then we have the following proposition.

\begin{proposition}\label{prp:Snpos}
Assume that $\balpha$ has at least two positive and two negative coordinates. Then Propositions~\ref{prp:Sn0} and~\ref{prp:Sn1} remain true if we replace $S$ with $S^+$ and $C$ with $C^+$.
\end{proposition}
\begin{proof}
To obtain a formula for $S_{n, 0}^+(H, J; \balpha)$, we use similar arguments as in the proof of Proposition~\ref{prp:Sn0}. The main differences are we only consider the case $\nu_i=1$ for an index $1\le i\le n$, and we replace the volume-counting expression. In particular, for a fixed $i$, by Theorem~\ref{thm:shifted-lattice-box-count} and Corollary~\ref{cor:Valpha-BJ-approx}, the number of such vectors $\bnu \in \cS_{n,0}^+(H,J;\balpha)$ is 
\begin{align*}
    \#\left\{\bnu \in (\Z \cap (0, H])^{n - 1} : \balpha^*_{i} \cdot \bnu = J - \alpha_{i}\right\}
    &= V_{\balpha^*_{i}}([0,H]^n;J-\alpha_i)+O(H^{n-3}) \\
    &=  
    V_{\balpha^*_{i}}([0,1]^n;0)H^{n-2}+O(H^{n-3}+|J|H^{n-3}),
\end{align*}
when $\balpha_i^*$ also have positive and negative coordinates and $\gcd(\balpha_i^*) \mid J-\alpha_i$. Adding these expressions over all $i$ gives the formula for $S_{n, 0}^+(H, J; \balpha)$.

We now obtain a formula for $S_{n, 1}^+(H, J; \balpha)$. We proceed similarly as in the proof of Proposition~\ref{prp:Sn1}, except that we only consider the case $\nu_{i_1} = \nu_{i_2}$. We then obtain an expression similar to~\eqref{eqn:tpmt-int2}, and adding these expressions over all $i_1, i_2$ gives the formula for $S_{n, 1}^+(H, J; \balpha)$.
\end{proof}

When (without loss of generality) all coordinates of $\balpha$ are positive, we define \begin{align*}
 &C^{+,(0)}_{\balpha;J}\coloneqq \dfrac{1}{(n-2)!} \sum_{i=1}^n  \dfrac{\delta_{\balpha_i^*}(J-\alpha_i) \gcd(\balpha_i^*)}{\prod_{j \ne i} \alpha_j},\\
 &C^{+,(1)}_{\balpha;J}\coloneqq \dfrac{1}{(n-2)!} \sum_{i=1}^n  \dfrac{\delta_{\balpha_{i_1,i_2}^+}(J) \gcd(\balpha_{i_1,i_2}^+)}{ (\alpha_{i_1}+\alpha_{i_2})\prod_{j \ne i_1,i_2} \alpha_j}.
\end{align*}

In this case, we have the following proposition.

\begin{proposition}\label{prp:allpos}
Assume that $\balpha$ has positive coordinates and $J > \max_i \alpha_i$. For $r=0,1$, 
\[ S_{n,r}^+(J;\balpha) = C^{+,(r)}_{\balpha;J} J^{n-2}+O(J^{n-3}). \]
\end{proposition}
\begin{proof} We argue similarly as in the proof of Proposition~\ref{prp:Snpos}, except that we compute $V_{\balpha}([0, cJ]^n; J)$ for $c = \max\{\alpha_i^{-1} : i \leq n\}$ using Corollary~\ref{cor:Valpha-BJ-pos}, rather than approximate using Corollary~\ref{cor:Valpha-BJ-approx}.
\end{proof}

We end this section by noting our arguments can be modified to obtain results where the vector $\balpha$ has exactly one negative (or positive) coordinate, or where its coordinates are all nonnegative with some being zero. However, due to technical considerations for deriving $C^{+,(0)}_{\balpha;J},C^{+,(1)}_{\balpha;J}$ in these cases, we omit the proof.

\section{Comments}\label{sec:remark}

\subsection{On the strength of the results}\label{sec:strength}
As mentioned previously, the order of our asymptotical formulae on $S(H,J;\balpha)$ is expected to be $H^{n-2}$, since the multiplicatively dependence condition and hyperplane condition are of one degree of freedom each. We also note that the main error terms of our formulae are of order $H^{n-5/2}$. Following our argument, especially in the proof of Proposition~\ref{prp:Sn1}, this error term  comes from vectors of rank 1 of the form \[
(\pm x,\pm x^2,\nu_3,\dots,\nu_n).
\] Since there are $H^{1/2}$ choices of $x$ and $H^{n-3}$ choices for the rest of the coordinates, this error term is naturally expected. With some additional arguments, this error term might be made explicit, and one may repeat the same arguments (when $k\ge 5$) for vectors of the form $(\pm x^a,\pm x^b,\dotsc)$, as long as $a,b\ll \log H$. In fact, in this case any contributions of order between $H^{n-3}$ and $H^{n-5/2}$ should only come from vectors of rank $1$, based on Corollary~\ref{cor:Snrupper}. We may compare this result with the results in~\cite[Section~1]{PSSS}, where the error terms in their results only comes from vectors with multiplicative ranks at least two.

We also believe that the log factor of the error terms in Theorems~\ref{thm:main-k>=3} and~\ref{thm:main-k=2} for $2\le k\le 4$ may be removed. In particular, Igor E. Shparlinski noted that this can be done if the curves in Theorems~\ref{thm:bompil-2var}-\ref{thm:bompil-4var} admit no quadratic factor with integer points. In this case, one needs to modify the secondary term of Proposition~\ref{prp:CCDN} to consider integer points from the related cubic curves more carefully. 

We now direct our attention to the constant $C_{\balpha;J}$ in Theorem~\ref{thm:main-k>=3} for the case $k \geq 4$. By observing our arguments in Section~\ref{sec:small-rank} to derive $C_{\balpha;J}$, we may see that the dependence of $J$ only comes from the fact that the corresponding point counting arguments may generate zero solutions depending on the value of $J$ relative to $\balpha$. Therefore, if we can guarantee that all related hyperplane equations admits an integer solution, the constant may be derived independently of $J$. More explicitly, from the proof of Propositions~\ref{prp:Sn0} and~\ref{prp:Sn1}, we may conclude that $C_{\balpha;J}$ is independent of $J$ if and only if all the equations
\begin{align}\begin{split}\label{eqn:alpha-system}
    \balpha_i^*\cdot \bnu = J-\alpha_i,\qquad
    \balpha_i^*\cdot \bnu = J+\alpha_i,\qquad
    \balpha_{i_1,i_2}^+ \cdot \bnu = J,\qquad
    \balpha_{i_1,i_2}^- \cdot \bnu = J,
\end{split}
\end{align}
    with $1\le i,i_1,i_2\le n$ and $i_1<i_2$, admits a solution $\bnu \in \Z^n$.
One instance is when $\balpha$ has three coordinates equal to $\pm 1$, since each $\balpha_i^*$ and $\balpha_{i_1, i_2}^{\pm}$ will have at least one coordinate equal to $\pm 1$.

We proceed to discuss the case where $C_{\balpha;J}$ is zero. From the proof of Propositions~\ref{prp:Sn0} and~\ref{prp:Sn1}, this happens when each equation in~\eqref{eqn:alpha-system} admits no solution. This is trivially true when $\gcd(\balpha) \nmid J$, and thus we will discuss some choices of $\balpha$ and $J$ such that $\gcd(\balpha) \mid J$ but $C_{\balpha;J}=0$.

Consider $n^2$ pairwise distinct primes, say $p_i$ for each $i = 1, 2, \ldots, n$, and $q_{i, j}$ and $r_{i, j}$ for each $1 \leq i < j \leq n$.
By Chinese remainder theorem, there exists an integer vector $\balpha = (\alpha_1, \alpha_2, \ldots, \alpha_n)$ with the following property:
\begin{itemize}
    \item   for each $i \leq n$, we have $\alpha_i \equiv 1 \pmod{p_i}$ and $\alpha_j \equiv 0 \pmod{p_i}$ if $j \neq i$;
    \item   for each $1 \leq i_1 < i_2 \leq n$, we have $\alpha_{i_1} \equiv 1 \pmod{q_{i_1, i_2}}$, $\alpha_{i_2} \equiv -1 \pmod{q_{i_1, i_2}}$, and $\alpha_j \equiv 0 \pmod{q_{i_1, i_2}}$ if $j \neq i_1, i_2$;
    \item   for each $1 \leq i_1 < i_2 \leq n$, we have $\alpha_{i_1} \equiv 1 \pmod{r_{i_1, i_2}}$, $\alpha_{i_2} \equiv 1 \pmod{r_{i_1, i_2}}$, and $\alpha_j \equiv 0 \pmod{r_{i_1, i_2}}$ if $j \neq i_1, i_2$.
\end{itemize}
Then none of the primes $p_i, q_{i_1, i_2}, r_{i_1, i_2}$ divides $\gcd(\balpha)$.
However, $p_i \mid \gcd(\balpha_i^*)$ for each $i \leq n$, and we have $q_{i_1, i_2} \mid \gcd(\balpha_{i_1, i_2}^+)$ and $r_{i_1, i_2} \mid \gcd(\balpha_{i_1, i_2}^-)$ for each $1 \leq i_1 < i_2 \leq n$.
As a result~\eqref{eqn:alpha-system} admits no solution when $J$ is not divisible by any of the $n^2$ distinct primes, and $C_{\balpha;J} = 0$.

We also note that we actually may let $H\ge 1$ in Theorem~\ref{thm:main-k>=3} for $k\ge 4$. The stricter restriction in the theorem statement only come from the case $k=3$, where this condition comes from the argument before~\eqref{eqn:k=3,r=2,balpha}. 

Finally, we discuss our results for the case $k=2$, with fixed $J$ and $H$. With regards to the arguments in Section~\ref{sec:k=2-main}, we need to find a formula for the number $\#(\cS_2'(J;\alpha_{j_1},\alpha_{j_2})\cap [-H,H]^2)$, where the definition can be seen in the corresponding section. In particular, we need to provide a uniform count on the number of solutions to~\eqref{eqn:multdep-k=2} with $|\nu_{j_1}|,|\nu_{j_2}|\le H$. While we are currently unable to obtain a result in this form, we suggest that a strategy similar to the one employed in~\cite[Section 4]{DS2018} may be worth exploring.

\subsection{On the conditions of Theorems~\ref{thm:bompil-2var}-\ref{thm:bompil-4var}}\label{sec:cond-bompil}We first note that our proof actually shows that one may replace the $k_1+k_2+k_3$ (or $k_1+k_2+k_3+k_4$, for Theorem~\ref{thm:bompil-4var}) terms in our results with $\max(k_1+k_2,k_3)$. However, we choose this weaker, but symmetric, bound in our use cases.

We note that all system of equations discussed in these propositions require $J\ne 0$. This condition is neccessary, as seen by the following examples. First, for Theorem~\ref{thm:bompil-2var}, we consider the system of equations
\begin{align*}
    \nu_1\nu_3=\nu_2,\qquad\alpha_1\nu_1+\nu_2=0,
\end{align*}
which admits solutions $(x,-\alpha x, -\alpha)$ for all $1<|x|\le H/\alpha$. Next, for Theorem~\ref{thm:bompil-3var}, we note that the system of equations \begin{align*}
    \nu_1^2\nu_2=\nu_3^3,\qquad
&2\nu_1+3\nu_2-13\nu_3=0
\end{align*}  admit solutions of the form $(x,8x,2x)$ for all integers $1<|x|\le H/8$. 
Finally, for Theorem~\ref{thm:bompil-4var}, we consider the system of equations\begin{align*}
    &\nu_1\nu_3=\nu_2\nu_4,\qquad
    \alpha_1\alpha_4=\alpha_2\alpha_3,\\
&\alpha_1\nu_1+\alpha_2\nu_2+\alpha_3\nu_3+\alpha_4\nu_4=0,
\end{align*} which admits solutions of the form $(\alpha_2x,-\alpha_1x,\alpha_3y,-\alpha_4y)$ for all integers $x,y$ and therefore gives $\gg H^2$ solutions $\bnu\in \Z^4\cap [-H,H]^4$.

Next, we return to Theorem~\ref{thm:bompil-3var} and discuss the neccessity  of the condition $\alpha_1\nu_1\ne J$. This can be seen by considering the system of equations \begin{align*}
    \nu_1\nu_2=\nu_3, \quad
J\nu_1+\nu_2-\nu_3=J,
\end{align*} which admits  $(1,x,x)$ for all integers $0<|x|\le H$ as solutions and therefore more than $H$ solutions in $[-H,H]^3$.

Finally, we discuss the strength of our propositions. First, we expect our error term in Theorem~\ref{thm:bompil-2var} is tight with respect to the main power of $H$ with the following argument. Picking any $|\nu_1| \ll H^{1/2}$ as $H\to \infty$ for~\eqref{eqn:fatal2var-linear}, we may guarantee that $\nu_2 =O(H^{1/2})$. Then we note that the tuples $(\nu_1,\nu_2,\nu_1\nu_2)\in [-H,H]^{3}$  satisfies~\eqref{eqn:fatal2var-1} with $k_1=k_2=k_3=1$. This ensures that the system of equations has $\gg H^{1/2}$ solutions $\bnu$, as $H\to \infty$. A similar construction also exists if we want to replace~\eqref{eqn:fatal2var-1} with~\eqref{eqn:fatal2var-2}.

On another hand, with respect to Theorem~\ref{thm:bompil-3var}, we note that Igor E. Shparlinski, from private communications with the first author, conjectured that the system of equations~\eqref{eqn:fatal1} and~\eqref{eqn:fatal2} should have $H^{o(1)}$ solutions $(\nu_1,\nu_2,\nu_3)\in (\Z\cap [-H,H])^3$ for most choices of $(A,B,\balpha,k_1,k_2,k_3)$. Indeed, if the related affine curve $f$ is irreducible, then one may apply Lemma~\ref{lem:bompil-CCDN} directly to $f$ and obtain that at most $O(H^{1/\max(k_1+k_2,k_3)}\log H)$ integer points lies in $f$, for a fixed $(k_1,k_2,k_3)$. However, proving an arbitrary affine curve $f$ is irreducible is generally hard, especially as $\deg f\to \infty$. In addition, one may also hope to use results on affine dimension growth as seen at~\cite{CCDN,CDHNV,Verm} for bounding the number of vectors of large rank.

\subsection{A correction to Pappalardi-Sha-Shparlinski-Stewart}\label{sec:cor-PSSS}

We use this section to correct a bound of Pappalardi-Sha-Shparlinski-Stewart that arose by substituting $d=1$ to~\cite[Equation (1.17)]{PSSS}. We use this section to correct an erroneous bound reported by Pappalardi, Sha, Shparlinski, and Stewart. The original assertion, obtained from \cite[Equation (1.17)]{PSSS} with $d=1$, incorrectly counted the number of multiplicatively dependent vectors in $\bnu\in (\Z \cap [-H,H])^2$ as $12H+O(1)$  as $H\to \infty$. We acknowledge that the potential error (a typing mistake in their Theorem 1.4) was first noted by Pineda-Ruelas~\cite{P-R} in their review of \cite{PSSS}.

To show that their argument is incorrect, we note that~\cite[Equation (1.17)]{PSSS} was derived by adding~\cite[Equation (4.18)]{PSSS} and~\cite[Equation (4.21)]{PSSS}, which has error terms of order $O(1)$ and $O(\log H)$ respectively. Therefore, the resulting error term be $O(\log H)$ instead of $O(1)$.

We now provide a heuristic on why this error bound cannot be improved. We first note that the main term is correct, as seen from pairs of the form \begin{equation}\label{eqn:n=2}
(x,\pm 1), (\pm1,x), (\pm |x|,\pm |x|) \quad \text{for} \quad 1<|x|\le H.
\end{equation} However, we note that there are at least $\log H/\log 2$ multiplicatively dependent vectors of the form $(2^a,2^b)$ for $1<a\ne b < \log H$. Thus, the number of multiplicatively dependent vectors $\bnu\in (\Z \cap [-H,H])^2$ that are not of form~\eqref{eqn:n=2} is $\gg \log H$. Hence, we conclude there are  \[
12H+O(\log H).
\]multiplicatively dependent vectors $\bnu\in (\Z \cap [-H,H])^2$.

\subsection{Multiplicatively dependent positive integers with fixed sum of coordinates}\label{sec:allone}

As a natural exposition for our result, this section discusses counting $S_{n}^+(H,J;\bfone_n)$, where $\bfone_n$ is the all-one vector. For simplicity, we denote 
\[ S_n^+(J)=S_n^+(H,J;\bfone_n) \qquad \text{and} \qquad S_{n,r}^+(J)=S_{n,r}^+(H,J;\bfone_n). \]  Now, counting $S_n^+(J)$ corresponds to counting solutions $\bnu\in \Z_+^n$ to \begin{equation}\label{eqn:sumcoordinates}
\nu_1+\dots+\nu_n=J
\end{equation} such that $\bnu$ is multiplicatively dependent. In this perspective, computing $S_n^+(J)$ is equivalent to counting multiplicatively dependent positive integer vectors whose coordinates sum to $J$, or counting compositions of $J$ as multiplicatively dependent tuples.  It is easy to see that
\[ C^{+,(0)}_{\bfone_n; J} = \frac{n}{(n - 2)!} \quad \text{and} \quad C^{+,(1)}_{\bfone_n; J} = \frac{n(n - 1)}{4(n - 2)!}, \]
    where $C^{+,(0)}_{\bfone_n; J}$ and $C^{+,(1)}_{\bfone_n; J}$ are as defined in Section~\ref{sec:positive}.
Thus, applying Theorem~\ref{thm:main-positive-integer} gives
\begin{equation}\label{eqn:bone}
    S_n^+(J) = \frac{n(n+3)}{4(n-2)!}J^{n-2} + \begin{cases}
         O(J^{n-5/2}(\log J)^{16}) &\quad\text{for} \quad 3\le n \le 4,\\
         O(J^{n-5/2}) &\quad\text{for}\quad n\ge 5.
     \end{cases}
\end{equation}
We now discuss an alternative combinatorial proof for~\eqref{eqn:bone}. Following Corollary~\ref{cor:Snrupper}, we only need to compute $S_{n,r}^+(J)$ for $r=0,1$.

First, to compute $S_{n,0}^+(J)$, we need to count the number of solutions to~\eqref{eqn:sumcoordinates} with  at least one of its coordinates is equal to $1$. We note that we may choose the coordinates in $n$ ways, and the remaining $n-1$ terms would sum to $J-1$. Then, we obtain
\[ \binom{J - 2}{n - 2} = \dfrac{1}{(n-2)!}J^{n-2} + O(J^{n-3}) \]
possible vectors $\bnu$ for each fixed coordinate. On the other hand, by principle of inclusion-exclusion, there are $O(J^{n-3})$ vectors that contains at least two coordinates equal to $1$. Therefore, by considering each fixed coordinates, we obtain 
    \begin{equation}\label{eqn:Sn+0} S_{n,0}^+(J) = \frac{n}{(n - 2)!} J^{n - 2} + O(J^{n - 3}). \end{equation}
    
Next, we compute $S_{n,1}^+(J)$. Following similar arguments from Proposition~\ref{prp:Sn1}, we obtain that the main term of this expression comes from vectors of~\eqref{eqn:sumcoordinates} such that it contains (at least) two identical coordinates and all of the coordinates are bigger than $1$. Without loss of generality, let $x_{n-1}=x_n$.
Then, our problem is counting positive integer solutions to \begin{equation}\label{eqn:sumcoordinates-rank1}
x_1+\dots+x_{n-2}+2x_{n-1}=J-n.
\end{equation}
Now, for a fixed $x_{n-1}=k$, the number of solutions to~\eqref{eqn:sumcoordinates-rank1} is exactly
$\displaystyle \binom{J - n - 2k - 1}{n - 3}$.
Therefore, as $k$ ranges over all possible values, the number of solutions to~\eqref{eqn:sumcoordinates-rank1} is \begin{align*}
    \sum_{k \ge 1} \binom{J - n - 2 k - 1}{n - 3}
    &= \frac{1}{2} \left[ \sum_{k \ge 3} \binom{J - n - k}{n - 3} + \sum_{k \ge 3} (-1)^{k - 1} \binom{J - n - k}{n - 3} \right] \\
    &=  \frac{1}{2}\left[ \binom{J - n -2}{n - 2} + \sum_{m \ge 2} \left( \binom{J - n - 2m + 1}{n - 3} - \binom{J - n - 2m}{n - 3} \right) \right]\\
    &= \frac{1}{2} \left[ \binom{J - n - 1}{n - 2} + \sum_{m \ge 2} \binom{J - n - 2m}{n - 4} \right] \\ 
    &= \frac{1}{2(n - 2)!} J^{n - 2} + O(J^{n - 3}) .
\end{align*}
Noting there are $\binom{n}{2}$ possible choices for the pair of coordinates, and following the same reasoning as the proof of Proposition~\ref{prp:Sn1}, we obtain     \begin{equation}\label{eqn:Sn+1} 
S_{n,1}^+(J)= \frac{n(n-1)}{4(n - 2)!} J^{n - 2} + \begin{cases}
         O(J^{n-5/2}(\log J)^{16}) &\quad \text{for} \quad 3\le n \le 4,\\
         O(J^{n-5/2}) &\quad\text{for}\quad n\ge 5.
     \end{cases} .
\end{equation}
Combining~\eqref{eqn:Sn+0},~\eqref{eqn:Sn+1} and Corollary~\ref{cor:Snrupper} completes the proof of~\eqref{eqn:sumcoordinates}.

We now consider a problem similar to the one in the previous subsection. This problem involves counting partitions of $J$ as $n$ multiplicatively dependent parts, denoted as $p^{md}(J)$. 
Similarly as in the previous section, we consider the ``multiplicative rank" of the related partition by using a similar definition as in Definition~\ref{def:multrank}. Using similar reasoning as in the previous section, we may conclude there are $O(J^{n-3})$ multiplicatively dependent partitions of rank more than $1$. Therefore, we may now count multiplicatively dependent partitions of rank $0$ and $1$, denoted as $p^{md}_{n,0}(J,n)$ and $p^{md}_{n,1}(J,n)$, respectively.

To compute $p^{md}_{n,0}(J,n)$, we note that this problem is equivalent to counting the number of partitions of $J$ such that at least one of its parts equal to $1$. Such partition corresponds to a partition of $J-1$ to $n-1$ parts, which implies \[
p^{md}_{n,0}(J,n)=p_{n-1}(J-1),
\] where $p_k(N)$ denotes the number of partitions of $N$ to $k$ parts.

Next, to compute $p^{md}_{n,1}(J,n)$, we first note that the number of partitions of rank $1$ without equal coordinates or more than two equal coordinates is at most $O(J^{n-5/2})$. Therefore, we may now count partitions of $J$ to $n$ parts of rank $1$ with equal coordinates. Using similar arguments as in~\eqref{eqn:sumcoordinates-rank1}, the number of such partitions is exactly
$ \displaystyle \sum_{a \ge 1} p_{n-2}(J - n - 2a)$,
which implies \[
p_{n}^{md}(J)= p_{n-1}(J-1)+\sum_{a \ge 1} p_{n-2}(J - n - 2a)+O(J^{n-5/2})
\]
Similar question may also be asked for the number of multiplicatively dependent partitions of $J$, which is  $\displaystyle \sum_n p_{n}^{md}(J)$.

To end this discussion, we note that Knessl and Keller~\cite{KK} claimed
\[ p_k(n) = \frac{1}{k! (k - 1)!} n^{k - 1} + O(n^{k - 2}), \]
however as noted by Andrews in the Mathematical Review~\cite{And}, this result only holds ``formally"; i.e. this formula does not hold asymptotically. The reader may consult Yavuz Oru\c{c}~\cite{YO} for the  best known upper and lower bounds for $p_k(n)$.

\subsection{Further directions}\label{sec:remark-further} \subsubsection{Multiplicatively dependent vectors of small dimensions}\label{sec:small-dim}We now discuss related questions for $n=2,3$. This project was initially inspired by the first author's reading of~\cite{PSSS} for the case $n=3$, which results in a problem for the Indonesian national high-school mathematical olympiad (Olimpiade Sains Nasional) in 2024. The  following translation is taken from~\cite{AOPS}.

\begin{question}[Question 2, Olimpiade Sains Nasional 2024, mathematics national round for high school students]\label{que:fatal}
    The triplet of positive integers $(a,b,c)$ with $a<b<c$ is called a \textit{fatal triplet} if there exist three nonzero integers $p,q,r$ which satisfy the equation $a^p b^q c^r = 1$.
As an example, $(2,3,12)$ is a fatal triplet since $2^2 \cdot 3^1 \cdot (12)^{-1} = 1$.
The positive integer $N$ is called \textit{fatal} if there exists a fatal triplet $(a,b,c)$ satisfying $N=a+b+c$.
\begin{itemize}
    \item Prove that $16$ is not fatal.
\item Prove that all integers bigger than $16$ which are not an integer multiple of $6$ are fatal.
\end{itemize}
\end{question}
The term ``fatal triplet", coined by the first author, was  inspired by the song ``Fatal" from GEMN.\footnote{The song was the opening song for the second season of the Oshi no Ko anime that premiered in July 2024.
The authors would like to thank Aka Akasaka, Oshi no Ko anime series, Dogakobo and GEMN for providing inspirations that led the first author to propose Question~\ref{que:fatal} which kickstarted this work. However, the authors concur that the manga was ended unsatisfactorily, to which they would not like to thank Aka Akasaka. In addition, all authors have agreed that Akane Kurokawa is the best girl character of the franchise. } 

The first part of Question~\ref{que:fatal} can be solved by bruteforcing all possible triplets $(a,b,c)$ with $a+b+c=16$. For the second part, we consider the following constructions:
\begin{align*}
    6k+2&=2+2k+4k, & 2^1(2k)^1(4k)^{-1}&=1,\\
    6k+4&=4+2k+4k, & 4^1(2k)^2(4k)^{-2}&=1,\\
    4k+3&=3+k+3k, & 3^1(k)^1(3k)^{-1}&=1,\\ 4k+9&=9+k+3k, & 9^1(k)^2(3k)^{-2}&=1, \quad \text{for} \quad k\neq 3,9.
\end{align*}
    For the remaining case ($4k+9$ with $k=3,9$), we consider
\begin{align*}
    21&=1+4+16=3+6+12, & 1^14^216^{-1}=3^16^{-2}12^1&=1, \\
    45&=6+12+27=3+12+24, & 6^6 12^{-3} 27^1 = 9^1 12^{6}24^{-4}&=1.
\end{align*} 
However, the authors are still unsure about the existence of fatal triples for $6 \mid N$.

With respect to our work, we note that all vectors $\bnu\in \Z^3$ of rank $2$ are fatal triplets. However, not all fatal triplets have rank $2$, as demonstrated by $(1,4,16)$. In addition, our bound in~\eqref{eqn:k=2, r=2} only shows that there is $O(H^{1/2}(\log H)^{9})$ such triplets of rank $2$ with all coordinates in $\bnu$ between $1$ and $H$. 
Therefore, it is natural to pose the following relevant question, which can also be asked for bigger dimensions. \begin{question}
    For an $\balpha \in \Z^3$ with nonzero coordinates and $J\in \Z$, does there always exist $\bnu\in \Z^3$ of rank $2$ with $
    \alpha_1\nu_1+\alpha_2\nu_2+\alpha_3\nu_3=J
    $?
\end{question}

We also discuss some related results to Lemma~\ref{lem:extra-count-k=2}.
First, Dubickas and Sha~\cite[Section~4]{DS2018}, who considered the case $(\alpha,\beta)=(1,1)$, used results from Bugeaud and Luca~\cite{BL2006} to obtain several results on this setup. 
In particular, with our notation, they proved that under the ABC conjecture~\cite[Theorem~4.8]{DS2018}, when $\balpha=(\pm1,\pm1)$, the number of multiplicatively dependent vectors $\bnu = (\nu_1,\nu_2)\in \Z^2$ with $\balpha\cdot \bnu=J$ is bounded above by an absolute constant $C$. Furthermore, for large enough $J$ one may take $C=9$. However, their unconditional result~\cite[Theorem~4.4]{DS2018} only implies that the number of solutions is $O(2^{\omega(J)})$, where $\omega(n)$ denotes the number of distinct prime factors of $n$. For general $\balpha$, one may hope that Scott and Styer's result~\cite{SS} on Diophantine equations of the form $\pm ra^x\pm sb^y=c$ be useful in this case.

Finally, we note that several other questions on linear translations of multiplicatively dependent vectors have been discussed in the literature, mostly for the case $n=2,3$.  For example, for the case $n=3$, the works~\cite{BIV,VZ} provides result on the finiteness of multiplicatively dependent tuples $(a,b,c)$ such that $(a+1,b+1,c+1)$ and $(a+2,b+2,c+2)$ are also multiplicatively dependent. Moreover, for a fixed $a$, the number of pairs $(b,c)$ satisfying the above condition with $a<b<c$ can be bounded effectively with respect to $a$.

More is known for the same problem when $n=2$,  thanks to LeVeque~\cite{LV} and Mihailescu~\cite{Mih}. We refer the reader to~\cite{VZ} for more discussions in this case. In addition, for algebraic integers one may want to consult~\cite{BMZ99,BMZ06,BMZ03,Maur}, as mentioned in the next section.

\subsubsection{Further extensions}\label{sec:further-ext}
A natural generalisation, following~\cite{PSSS}, is counting vectors on a hyperplane whose coordinates are algebraic integers or numbers of fixed degree, or within a fixed number field, and bounded height. Our argument in Corollary~\ref{cor:Snrupper}, which is an extension of~\cite[Propositions 1.3 and 1.5]{PSSS} works well to obtain a similar result for bounding vectors of large multiplicative rank. However, in order to obtain asymptotic formulae for this setup, one may hope to derive a volume-counting formula as seen in  Section~\ref{sec:prelim-lat}. In addition, for small $k$, one also needs an analogue of Bombieri-Pila results for number fields, as seen at Section~\ref{sec:bompil}.

A particularly interesting setup in this direction is when $\bnu\in \Q^n$. In this case, we count the number of multiplicatively dependent vectors $\bnu$ on a fixed hyperplane and whose coordinates are Farey fractions $a_i/b_i$ of height at most $H$. We may obtain bounds on the hyperplane equation based on Shparlinski's result~\cite{Shp} on counting Farey fractions on a hyperplane. However, it remains to be seen whether combining this bound with the multiplicative relation will yield new results. Similar questions arise if we replace Farey fractions with Egyptian fractions, where one uses the result from the work of the first author with Kuperberg, Ostafe and Shparlinski~\cite[Lemma 3.5]{AKOS} to obtain the preliminary bound.

Another possible direction with respect to the propositions in Section~\ref{sec:bompil} is to count the number of solutions to the related equations for fixed exponents $k_1$, $k_2$, $\dotsc$, $k_n$. One might be able to use~\cite{AC} to count the number of solutions to the related multiplicative equation and combine this with the hyperplane (or other varieties) equation. We also note that~\cite{BMZ99} consider this problem for $n=2$ over more general settings, namely a curve $C\in \G_{m}^n$ with the family of proper algebraic subgroups of  $\G_m^n$, which was further generalised at~\cite{BMZ03,BMZ06,Maur}.

Yet another possible direction is to work with more general varieties. New results could be obtained by applying our methods to varieties arising from intersections of hyperplanes. Nevertheless, attempting a general formulation in this setup would lead to technical complications, especially while considering the related coefficients of the hyperplanes.
For varieties of higher degree, one may replace Bombieri-Pila's counting arguments with arguments based on Schwartz-Zippel principles, which may possibly give new results. Bombieri, Masser and Zannier's 2008 result~\cite{BMZ08} should also be useful in this case.

\section*{Acknowledgements}

The authors would like to thank Alina Ostafe, Igor E. Shparlinski and Jerry Wang for supports and suggestions during the preparation of this work. The authors would also like to thank Chip Corrigan and Ingrid Vukusic for helpful discussions related to the multiplicative equation. The authors would also like to thank Anders Mah for his suggestion to improve Proposition~\ref{prp:CCDN}. MA is supported by  Australian Research Council Grants  DP230100530 and a UNSW Tuition Fee Scholarship.

\end{document}